\documentclass[11pt]{amsart}

\usepackage{amsmath, amssymb, amsthm, amsfonts}
\usepackage{latexsym}
\usepackage{graphicx}
\usepackage{enumerate}
\usepackage{tikz}
\usepackage{comment}
\usetikzlibrary{calc}
\usetikzlibrary{shapes.geometric, patterns}

\usepackage{soul}
\def\rupert#1{\sethlcolor{cyan}\hl{#1}}

\def\epsilon{\varepsilon}

\newtheorem{Theorem}{Theorem}[section]
\newtheorem{Proposition}[Theorem]{Proposition}
\newtheorem{Lemma}[Theorem]{Lemma}
\newtheorem{Corollary}[Theorem]{Corollary}
\newtheorem{Example}[Theorem]{Example}

\newtheorem{Fact}[Theorem]{Fact}
\theoremstyle{remark}

\theoremstyle{definition}
\newtheorem{Definition}[Theorem]{Definition}

\def\srwn{\mathrel{\lneqq\!\raisebox{-0.1em}{\scalebox{0.6}{\!${\substack{sW\\W\phantom{s}}}$}}}}
\def\srbn{\mathrel{\lneqq\!\raisebox{-0.1em}{\scalebox{0.6}{\!${\substack{sB\\B\phantom{s}}}$}}}}
\def\v#1#2{\ensuremath{_{\vec{#1}}^{{\vec{#2}}}}}
\def\v#1#2{\ensuremath{_{\vec{#1}}^{{\vec{#2}}}}}
\def\uh{{\upharpoonright}}

\begin{document}
\title[The computational content of multidimensional discontinuity]{The computational content of\\multidimensional discontinuity\\{\scriptsize (Draft version)}
}

\author[R.\ H\"olzl]{Rupert H\"olzl}
\address{Universität der Bundeswehr München}
\email{r@hoelzl.fr}

\author[K.M.\ Ng]{Keng Meng Ng}
\address{Nanyang Technological University}
\email{kmng@ntu.edu.sg}
\thanks{Ng was supported by the Ministry of Education, Singapore, under its Academic Research
Fund Tier 2 (MOE-T2EP20222-0018).}

\date{\today}

\begin{abstract}	
	The Weihrauch degrees are a tool to gauge the computational difficulty of mathematical problems. Often, what makes these problems hard is their discontinuity. We look at discontinuity in its purest form, that is, at otherwise constant functions that make a single discontinuous step along each dimension of their underlying space. This is an extension of previous work of Kihara, Pauly, Westrick from a single dimension to multiple dimensions. Among other results, we obtain strict hierarchies in the Weihrauch degrees, one of which orders mathematical problems by the richness of the truth-tables determining how discontinuous steps influence the output.
\end{abstract}

\maketitle

\section{Introduction}

	The Weihrauch degrees are a tool to gauge the computational content of mathematical theorems and have been applied successfully to a large number of mathematical problems. Here, a problem always is supplied in the form of an {\em instance} and the task is to find one of its admissible {\em solutions}. Problems that have been successfully analysed using this approach come from analysis, algebra, probabilistic computation, recursion theory, and numerous other mathematical fields; for an overview of the relevant definitions and most important results, we refer the reader to a recent survey by Brattka, Gherardi, and Pauly~\cite{DBLP:journals/corr/BrattkaGP17}.
	
	To classify different levels of difficulty of problems, we require a way to compare their difficulty with each other. This is done using the notion of Weihrauch reducibility. This works by imagining a mathematical problem as a black box to which an instance is provided (in some suitable encoded form). The black box then is imagined to return one of the admissible solutions (again encoded in some form). With this paradigm in mind, we can now ask what such a black box can be used for, in other words: Assuming that we have a black box
	that solves all valid instances of a certain problem $\mathcal{A}$, can  it be repurposed to solve all valid instances of a different problem $\mathcal{B}$?
	
	As valid $\mathcal{A}$-instances may look very different from valid $\mathcal{B}$-instances, we need to translate between both problems and their solutions. To do this, an effective {\em pre-processing} procedure (essentially a Turing functional) is used to translate any valid $\mathcal{B}$-instance into a valid $\mathcal{A}$\nobreakdash-instance. Then we can feed this into the black box solving $\mathcal{A}$. Another effective procedure, the {\em post-processing}, then translates this solution for the   $\mathcal{A}$-instance into a solution to the original $\mathcal{B}$-instance.
	Note the uniformity, that is, that the functionals must not depend on the individual instances.	

	The reducibility induces a rich lattice of computational problems of different difficulties. One of the biggest achievements in this area is the identification of the class of so-called choice problems. They encapsulate in some sense the cores of difficulty hidden in many mathematical problems. This way they can be said to form a scaffolding underlying and structuring the entire lattice. For an overview we again refer the reader to the survey by Brattka, Gherardi, and Pauly~\cite{DBLP:journals/corr/BrattkaGP17}.
	
	In this article we investigate the hurdles to algorithmically solving mathematical problems using a different approach. Typically, what makes these problems interesting and hard is their discontinuity.	Here, we look  systematically at discontinuity in its purest isolated form, that is, at otherwise constant functions that make a single discontinuous step along each dimension of their underlying space. We will mostly focus on the multidimensional case, extending previous unpublished work of Kihara, Pauly, and Westrick~\cite{LinTakArn}, presented by Westrick~\cite[Section 5.10]{brattka_et_al:DagRep.8.9.1}, that focussed on a single dimension.
	
	We will work in Cantor space $2^\omega$ and use the lexicographic ordering~$<_{lex}$. Following Kihara, Pauly, and Westrick~\cite{LinTakArn}, who in turn were extending work of Day, Downey, and Westrick~\cite{day2022three}, for $\alpha\in 2^\omega$ we define $s_\alpha:2^\omega\rightarrow \{0,1\}$ by 
	\[s_\alpha(x)=\begin{cases}0,&\text{if $x<_{lex}\alpha$},\\1,&\text{if $x\geq_{lex}\alpha$.}\end{cases}\]
	It is worth mentioning that we could have alternatively worked on the unit interval of the real numbers instead of Cantor space, and in fact most of the results in this paper hold for both collections of Weihrauch degrees 
	\[\mathcal{S}_{2^\omega}=\{ \deg_W(s_\alpha)\mid\alpha\in 2^\omega\} \quad \text{and}\quad \mathcal{S}_{[0,1]}=\{ \deg_W(s_x)\mid x\in [0,1]\}.\]
	However, it can be shown that $\mathcal{S}_{2^\omega}\cap \mathcal{S}_{[0,1]}=\{\deg_W(s_0)\}$ and thus the two sets of Weihrauch degrees are virtually disjoint.
	
	Focussing again on Cantor space, we can give a first interesting example of a non-trivial degree describable by a discontinuous step function.
	\begin{Example}
		By using a bit-flipping pre-processing function, it is easy to see that $s_{1^\omega} \equiv_{sW} \mathrm{LPO}$.
	\end{Example}
	
	\medskip
	
	The remainder of the article is structured as follows: In Section~\ref{ndasbjhawehjvxsdabhcvxd} we make an interesting new observation for the one-dimensional case, namely that there exist $1$-equivalent c.e.\ sets $A$ and $B$ such that
	$s_\alpha \mid_W s_\beta$.
	
	In Section~\ref{stepintwodim}, we begin by considering ways of generalizing the one-dimen\-sional notion~$s_\alpha$. While the possibly most obvious idea for doing so might be to allow \textit{multiple} discontinuous steps, we give an easy argument why this is of little independent interest; and this provides the motivation for studying the multidimensional case instead where after evaluating each discontinuity a truth table is applied to the resulting binary vector to obtain the final binary output. 
	Initially, we allow the thresholds for the discontinuous jumps to be different along each dimension and show that this makes the resulting mathematical problems incomparable in the Weihrauch degrees, except for cases where they are comparable trivial reasons. 
	
	In Section~\ref{jqewhfjksadjkvf}, we study how the special case where the truth table is the parity function compares to the Weihrauch degree of products of multiple copies of the one-dimensional case; this leads to the identification of the first hierarchy of degrees.
	
	In Section~\ref{hhdashdfhbjshjehjchjasd}, we then fix all discontinuity threshold to the same value (or, equivalently, to a vector of pairwise Weihrauch equivalent values). We identify a criterion for quantifying the richness of different truth tables, and show that this exactly captures the strength of the defined mathematical problem in the Weihrauch degrees.
		
	\section{Contrasting with the $1$-degrees}\label{ndasbjhawehjvxsdabhcvxd}
	
	Before moving on to the multidimensional case, we will make a new observation concerning the one-dimensional case.
	\begin{Definition}
		We call $\alpha\in 2^\omega$ proper if it contains infinitely many $1$'s.
	\end{Definition}
	In the following we write $\alpha \sqsubseteq \beta$ if $\alpha$ is a prefix of $\beta$; $\alpha \sqsubset \beta$ if it is a proper prefix; and use ``$*$'' to denote the concatenation of binary strings. Note that $s_\alpha((\alpha\uh n)*1^\omega)=1$ holds for arbitrary $\alpha$ and $n$; and for proper $\alpha$ we also have ${s_\alpha((\alpha\uh n)*0^\omega)=0}$. We will use the following fact that was already observed by Kihara, Pauly, and Westrick~\cite{LinTakArn}.
	\begin{Fact}\label{fact:boundary}
Fix some proper~$\alpha$ and suppose that $s_\alpha\leq_W s_\beta$ via pre-processing function~$\Phi$ and post-processing function~$\Psi$. Then $\Phi(\alpha)=\beta$. In particular, if we even have $s_\alpha\leq_{sW} s_\beta$, then $\Psi$~is the identity.
\end{Fact}
\begin{proof}
Suppose that there exists a least $m$ such that $\Phi(\alpha)(m)\neq\beta(m)$. 
Then, since $s_\alpha(\alpha)=1$, there must be an $M>m$ and a $j\in\{0,1\}$ such that $\Phi(\alpha\uh M)(m){\downarrow}=j$ and $\Psi(\alpha\uh M,j){\downarrow}=1$. 
As $\alpha$ is proper, there is some $\alpha'$ to the left of $\alpha$ which agrees with $\alpha$ up to length $M$. Then 
\[\Phi(\alpha')\uh (m+1)=\Phi(\alpha)\uh (m+1)=\left(\beta\uh m \right)* j\]
and $\Psi(\alpha'\uh M,j){\downarrow}=1$, which is a contradiction since $s_\alpha(\alpha')=0$.
\end{proof}
\begin{Proposition}
	There exist c.e.\ sets $A$ and $B$ such that
	\[A\equiv_1 B \quad \text{and} \quad s_\alpha \mid_W s_\beta,\]
	where $\alpha=\chi_A$ and $\beta=\chi_B$.
\end{Proposition}
\begin{proof}
We will build a c.e.~set $A$ and let $B(2n)=A(2n+1)$ as well as ${B(2n+1)=A(2n)}$ for all $n \in \omega$. Then clearly~$A\equiv_1 B$. To make both $\alpha$ and $\beta$ proper, we set $A(4n+2)=A(4n+3)=1$ for every~$n \in \omega$. All other values of $A$ and $B$ will be determined during the construction; we will need to satisfy the requirements
\begin{align*}
 R_{2e}{:}\;\;&  (\Phi_e,\Psi_e)\text{ does not witness $s_\alpha\leq_W s_\beta$,}\\
R_{2e+1}{:}\;\;& (\Phi_e,\Psi_e)\text{ does not witness $s_\beta\leq_W s_\alpha$.}
\end{align*}
Here $(\Phi_e,\Psi_e)$ is the $e$-th pair of Turing functionals, where $\Phi_e$ is a potential pre-processing function and $\Psi_e$ is a potential post-processing function. For each stage $s$, we let $\alpha_s$ and $\beta_s$ denote the current approximations of $\alpha$ and~$\beta$, respectively. The proof is a finite injury argument; the priorities of the requirements are arranged in the order $R_0<R_1<R_2<\dots$ 

\smallskip

At stage $s=0$ we let $A_s=\emptyset$ and let $m_{e,s}=4e$ for each $e$.

\smallskip
 
Now assume that $s>0$. For an even number $2e<s$, we say that $R_{2e}$ requires attention at stage $s$ if $m_{2e}$ and $m_{2e}+1$ are both not in $A$, and the following two conditions hold:
\begin{enumerate}[(i)]
\item $\Phi_e\left(\left(\alpha_s\uh m_{2e}\right) * 01*1^s\right)$ has produced at least~$m_{2e}+2$ bits of output.
\item Either 
\[\Phi_e\left(\left(\alpha_s\uh m_{2e}\right) * 01*1^s\right)\not\sqsupseteq\left(\beta_s\uh m_{2e}\right)*10,\]
or 
\[\Psi_e\left(\left(\alpha_s\uh m_{2e}\right) * 01*1^s,1\right){\downarrow}=1.\]
\end{enumerate}
An analogous definition holds for $R_{2e+1}$ with $\alpha_s$ and $\beta_s$ interchanged.


At stage $s>0$, if no $R_k$ with $k<s$ requires attention, do nothing. Otherwise fix the least such $k$ and assume it is of the form $k=2e$.
\begin{itemize}
\item If $\Phi_e\left(\left(\alpha_s\uh m_k\right) * 01*1^s\right)\not\sqsupseteq\left(\beta_s\uh m_k\right)*10$, then enumerate $m_k+1$ into~$A$ (and thereby cause $m_k\in B$);

\item if $\Phi_e\left(\left(\alpha_s\uh m_k\right) * 01*1^s\right)\sqsupseteq\left(\beta_s\uh m_k\right)*10$, then enumerate $m_k$ into~$A$ (and thereby cause $m_k+1\in B$). 
\end{itemize}
In either case, in addition, enumerate all $x$'s such that $m_k+1<x<s$ into $A$ (and thereby also into $B$). Then initialize all $m_j$'s with~$j>k$ by setting them to fresh, strictly increasing values that are multiples of $4$.

This completes the construction; in case that $k$ is of the form $k=2e+1$, we do the construction with $A$ and $B$ interchanged. 

\medskip

We verify that the construction works as required. For each~${k\in\omega}$, if $m_k$ is not initialized, then $R_k$ can be attended to at most one more time; therefore the value of $m_k$ must eventually reach a limit at some stage~$s_k$. We check that $R_k$~is eventually satisfied; w.l.o.g.\ assume that $k=2e$ and that $(\Phi_e,\Psi_e)$ witnesses~${s_\alpha\leq_W s_\beta}$. 

First, we claim that $R_k$ will receive attention at some stage~$t>s_k$. Assume otherwise and let  $s\geq s_k$ be large enough so that ${\alpha_s\uh m_k=\alpha\uh m_k}$; then by construction ${\beta_s\uh m_k=\beta\uh m_k}$ holds as well. As $R_k$~has not received attention after stage~$s_k$ we have 
\[
	{\alpha\uh (m_k+2)=\left(\alpha_s\uh m_k\right)*00} \quad \text{and} \quad {\beta\uh (m_k+2)=\left(\beta_s\uh m_k\right)*00}.
\]
By Fact \ref{fact:boundary}, we have that $\Phi_e(\alpha)=\beta$, and thus (i)~must eventually hold. Then we must have 
\[{\Phi_e\left((\alpha_s\uh m_k)*01*1^\omega\right)\sqsupset (\beta_s\uh m_k)*10},\]
as otherwise~(ii) would hold as well. Consequently, 
\[s_\beta\left(\Phi_e\left((\alpha_s\uh m_k)*01*1^\omega\right)\right)=1.\]
We also have ${s_\alpha\left((\alpha_s\uh m_k)*01*1^\infty\right)=1}$. Thus it follows that \[\Psi_e\left(\left(\alpha_s\uh m_{k}\right) * 01*1^s,1\right)=1,\] which means that (ii)~must hold after all, a contradiction.

Thus, $R_k$ will receive attention at some stage~$t>s_k$; there are two cases to consider. First, if 
\[\Phi_e\left((\alpha_t\uh m_k)*01*1^t\right)\not\sqsupseteq (\beta_t\uh m_k)*10,\] then, since by choice of~$s_k$ we have that $\alpha_t$ and $\beta_t$ will not change anymore on their first $m_k$ bits, it holds that
\[\alpha\sqsupset (\alpha_t\uh m_k)*01*1^t \quad \text{and} \quad \beta\sqsupset (\beta_t\uh m_k)*10.\]
Therefore $\Phi_e(\alpha)\neq\beta$, a contradiction. 

In the second case, we have
\[\Phi_e\left((\alpha_t\uh m_k)*01*1^t\right)\sqsupset (\beta_t\uh m_k)*10\; \text{and} \; \Psi_e\left((\alpha_t\uh m_k)*01*1^t,1\right)=1.\] Then the construction ensures that
\[\alpha\sqsupset (\alpha_t\uh m_k)*10*1^t \quad \text{and} \quad \beta\sqsupset (\beta_t\uh m_k)*01.\]
Therefore,
\[{s_\alpha\left((\alpha_t\uh m_k)*01*1^\omega\right)=0} \quad \text{and}\quad {s_\beta\left(\Phi_e\left((\alpha_t\uh m_k)*01*1^\omega\right)\right)=1}.\] 
But this would mean that the post-processing $\Psi_e$ does not act correctly on input $(\alpha_t\uh m_k)*01*1^\omega$, again a contradiction.
\end{proof}

\section{Multidimensional step functions with different thresholds}\label{stepintwodim}

A natural generalization of the step functions studied above seems to be the study of such functions with multiple points of discontinuity. However, it turns out that this is not a very interesting notion by the following proposition.
\begin{Definition}
	Let $(\alpha_1,\dots,\alpha_n) \in (2^\omega)^n$. Write $s_{(\alpha_1,\dots,\alpha_n)}$ for the function
	\[
	x \mapsto\begin{cases}
		0 & \text{if } x < \alpha_1,\\
		1 & \text{if } \alpha_1 \leq x < \alpha_2,\\
		\;\!\vdots&\;\;\;\;\;\;\;\;\;\;\;\;\;\!\vdots\\
		n-1 & \text{if } \alpha_{n-1} \leq x < \alpha_n,\\
		n & \text{if }  \alpha_n \leq x.
	\end{cases}
	\]
\end{Definition}
\begin{Proposition}
	For all $n\in \omega$ and $(\alpha_1,\dots,\alpha_n) \in (2^\omega)^n$ we have
	\[s_{(\alpha_1,\dots,\alpha_n)} \leq_{sW} \bigsqcup_1^n  s_{\alpha_i} \quad \text{and} \quad\, \bigsqcup_1^n  s_{\alpha_i} \leq_{W} s_{(\alpha_1,\dots,\alpha_n)}.\]
\end{Proposition}
\begin{proof}[Proof sketch]
	For the first statement, enclose each $\alpha_i$ by an open interval with endpoints that have finite binary expansion in such a way that each interval contains exactly one $\alpha_i$ and overlaps with its two neighbouring intervals on the left and on the right. The pre-processing then reads enough of the input to locate it inside one of these intervals, and then requests an answer from the corresponding $s_{\alpha_i}$.
	
	For the second statement, we are given as input a pair of a natural number~$k$ and an $x\in2^\omega$. The pre-processing is the identity function, and the post-processing outputs $1$ iff the return value from the black box is $k$ or larger.
\end{proof}
Another possible route for generalization is to go to multiple dimensions with a discontinuity threshold for each dimension. Then for a given point in multidimensional space we get multiple bits of output that describe in each dimension on which side of the respective threshold the point lies. We can then process this sequence of bits in different ways, namely through different truth-tables to produce a final output bit.
In this section, we will show that if we make no further assumptions about the thresholds besides properness, then in general a rather chaotic picture results in the Weihrauch degrees. This will then justify making stronger assumptions in the next section to make a more interesting picture emerge.
\begin{Definition}
	Fix an integer $n>1$ and proper $\alpha_1, \dots, \alpha_n \in 2^\omega$. Let $F$~be any $n$-dimensional truth-table, that is, any total function from $\{0,1\}^n$ to~$\{0,1\}$. We write $n(F)$ for $n$. Define the function $s^F_{\alpha_1,\dots,\alpha_n}\!\colon (2^\omega)^n \rightarrow 2$ via
	\[(x_1,\dots,x_n) \mapsto F\left(s_{\alpha_1}(x_1),\dots,s_{\alpha_n}(x_n)\right).\]
\end{Definition}
\begin{Definition}
	For $1 \leq i \leq n$ and $b\in\{0,1\}$, define $p_i^b\colon 2^n \rightarrow 2^n$ via
\[p_i^b(x_1,\dots,x_n)=(x_1,\dots,x_{i-1},b,x_{i+1},\dots x_n).
\]In other words, $p_i^b(x_1,\dots,x_n)$ outputs the vector with all entries equal to the input except for the $i$-th entry which is replaced with $b$.
	
	\medskip
	
	\noindent For an $n$-dimensional truth-table $F$, and $\vec{i}=(i_1,\dots,i_k) \in \{1,\dots,n\}^k$ such that the $i_j$'s are pairwise different, and $\vec{b}=(b_1,\dots,b_k) \in \{0,1\}^k$, we write $F\v{i}{b}$ for
	$F \circ p_{i_1}^{b_1} \circ \dots \circ p_{i_k}^{b_k}$.
\end{Definition}
It is natural to ask for which $n$\nobreakdash-dimensional truth tables $F$ and $G$ and which sequence of threshold values ${\alpha_1,\dots,\alpha_n\in 2^\omega}$ we have
\[s^G_{\alpha_1,\dots,\alpha_n}\leq_{sW} s^F_{\alpha_1,\dots,\alpha_n} \quad \text{or at least} \quad s^G_{\alpha_1,\dots,\alpha_n}\leq_W s^F_{\alpha_1,\dots,\alpha_n}.\]
If $G=F$ or $G=1-F$ then this is trivially the case; in fact, if
\[G=F\v{i}{b} \quad \text{or} \quad G=1-F\v{i}{b}\]
for some $\vec{i}$ and $\vec{b}$ then $s^G_{\alpha_1,\dots,\alpha_n}\leq_{sW} s^F_{\alpha_1,\dots,\alpha_n}$ is witnessed by a straight-forward choice for the pre-processing function; we give the details below. The following theorem states that it is possible to find threshold values $\alpha_1,\dots,\alpha_n\in 2^\omega$ such that the Weihrauch degrees corresponding to truth tables~$F$ and $G$ can only be comparable to each other for these trivial reasons.
\begin{Theorem}\label{thm:differentalpha}
	There is an infinite sequence $\vec{\beta} =(\beta_1,\beta_2,\dots)\in (2^\omega)^\omega$ of proper sequences  such that for all $n>0$ and all $n$-dimensional truth-tables~$F$ and $G$, the following are equivalent:
	\begin{enumerate}
		\item $s_{\vec{\beta}\uh n}^G \leq_{sW} s_{\vec{\beta}\uh n}^F$.
		\item $s_{\vec{\beta}\uh n}^G \leq_{W} s_{\vec{\beta}\uh n}^F$.
		\item There exist $\vec{i}=(i_1,\dots,i_k) \in \{1,\dots,n\}^k$ such that the $i_j$ are pairwise different and $\vec{b}=(b_1,\dots,b_k) \in \{0,1\}^k$, with $G =F\v{i}{b}$ or $G=1- F\v{i}{b}$.
	\end{enumerate}
Furthermore we can choose $\vec{\beta} $ so that $\beta_j\equiv_1\beta_k$ for every $j,k$.
\end{Theorem}
\begin{proof}
(1) $\Rightarrow$ (2) is trivial.

\medskip

\noindent (3) $\Rightarrow$ (1): Suppose that $G =F\v{i}{b}$ and fix any sequence $\vec{\beta}$. We let the pre-processing $\Phi(\gamma_1,\dots,\gamma_n)=(\delta_1,\dots,\delta_n)$, where
\[\delta_j=\begin{cases}
0^\omega, \text{ if there is an $\ell$ such that $j=i_\ell$ and $b_\ell=0$,} \\
1^\omega, \text{ if there is an $\ell$ such that $j=i_\ell$ and $b_\ell=1$,}\\
\gamma_j, \text{ otherwise},
\end{cases}\]
and use the identify function as the post-processing. Then for any~$(\gamma_1,\dots,\gamma_n)$,
\begin{align*}
	&s_{\vec{\beta}\uh n}^G(\gamma_1,\dots,\gamma_n)\\
	=\;&G\left(s_{\beta_1}(\gamma_1),\dots,s_{\beta_n}(\gamma_n)\right)\\
	=\;&F\v{i}{b}\left(s_{\beta_1}(\gamma_1),\dots,s_{\beta_n}(\gamma_n)\right)\\
	=\;&F\left(s_{\beta_1}(\delta_1),\dots,s_{\beta_n}(\delta_n)\right)\\
	=\;&s_{\vec{\beta}\uh n}^F(\delta_1,\dots,\delta_n)\\
	=\;&s_{\vec{\beta}\uh n}^F(\Phi(\gamma_1,\dots,\gamma_n)).
\end{align*}
Hence we have ${s_{\vec{\beta}\uh n}^G \leq_{sW} s_{\vec{\beta}\uh n}^F}$. If $G =1-F\v{i}{b}$ then use the same pre-processing~$\Phi$ as above together with the post-processing function that maps $0$ to $1$ and $1$ to $0$.

\medskip

\noindent $\neg(3) \Rightarrow \neg(2)$: We construct the sequence $\vec{\beta} =(\beta_1,\beta_2,\dots)$ simultaneously by finite extensions. Consider a list of all quadruples $(G,F,\Phi,\Psi)$ where $\Phi$~and~$\Psi$ are Turing functionals and $G$ and $F$ are $n$-dimensional truth tables for some~$n>1$ for which~(3) does not hold. 
	
Initially we let $\sigma_k[0]$ be the empty string for all $k$. At each stage $s+1$ with $s \geq 0$, 
for each $k \in \omega$ we have some $\sigma_k[s]$ given and need to find a suitable extension $\sigma_k[s+1]\sqsupset\sigma_k[s]$. This will then allow us to let $\beta_k=\cup_s\sigma_k[s]$ for each $k \in \omega$.

We act for the $s$-th quadruple $(G,F,\Phi,\Psi)$. Let $n=n(G)=n(F)$. The values of $\beta_k$ for $k>n$ are obviously irrelevant for $G$ and $F$, so we simply let $\sigma_k[s+1]=\sigma_k[s]*000$ for each $k>n$. For $k\leq n$, we choose $\sigma_k$ as follows:

Consider if there exists an $m\in \omega$ such that
\begin{itemize}
\item[(i)] for $(\tau_1,\tau_2,\dots,\tau_n):=\Phi(\sigma_1[s]*01^m,\sigma_2[s]*01^m,\dots,\sigma_n[s]*01^m)$ we have $\min_{j\leq n}\{|\tau_j|\}\geq\max_{j\leq n}\{|\sigma_j[s]|\}+3$, and
\item[(ii)] $\Psi(\sigma_1[s]*01^m,\sigma_2[s]*01^m,\dots,\sigma_n[s]*01^m,b)$ is defined for both $b=0$ and $b=1$, but with unequal outputs.
\end{itemize}
There are several cases to consider. If no $m$ satisfying (i) exists then we let $\sigma_k[s+1]=\sigma_k[s]*000$ for every $k\leq n$. If no $m$ satisfying~(ii) exists then there is some $b\in\{0,1\}$ such that 
\begin{align*}
&b\neq  \Psi(\sigma_1[s]*01^\omega,\sigma_2[s]*01^\omega,\dots,\sigma_n[s]*01^\omega,0)\\
\text{and}\quad & b\neq  \Psi(\sigma_1[s]*01^\omega,\sigma_2[s]*01^\omega,\dots,\sigma_n[s]*01^\omega,1).\phantom{\text{and}\quad}
\end{align*}
Since condition (3) of the theorem does not hold for the pair $(G,F)$, $G$ is not a constant function and therefore there is some tuple of truth values $(x_1,\dots,x_n)$ such that $G(x_1,\dots,x_n)=b$. Now for each $k\leq n$ we take $\sigma_k[s+1]=\sigma_k[s]*000$ if $x_k=1$ and take $\sigma_k[s+1]=\sigma_k[s]*100$ otherwise.

Finally assume that the least $m$ satisfying (i) and (ii) exists. First assume 
\begin{align*}
	& \Psi(\sigma_1[s]*01^m,\sigma_2[s]*01^m,\dots,\sigma_n[s]*01^m,0)=0\\
	\text{and}\quad & \Psi(\sigma_1[s]*01^m,\sigma_2[s]*01^m,\dots,\sigma_n[s]*01^m,1)=1.\phantom{\text{and}\quad}
\end{align*}
Consider the set of all numbers $k\leq n$ such that $\tau_k$ is either to the left of $\sigma_k[s]*001$ or to the right of
$\sigma_k[s]*100$. Let $i_1<i_2<\dots<i_{u_0}$ be a listing of all such $k$ (this sequence could of course be empty), and for each $u\leq u_0$, let $b_u=0$ if $\tau_{i_u}$ is left of $\sigma_{i_u}[s]*001$, and let $b_u=1$ otherwise. Let $\vec{b}=(b_1,\dots,b_{u_0})$. Now by the assumption on $G$ and $F$, there is some tuple of truth values $(x_1,\dots,x_n)$ such that $G(x_1,\dots,x_n)\neq F^{\vec{b}}_{\vec{i}}(x_1,\dots,x_n)$.
To fix $\sigma_k[s+1]$ for $k\leq n$, we proceed as follows:
\begin{itemize}
\item[(iii)] If $k\neq i_u$ for any $u$ then let $\sigma_k[s+1]=\sigma_k[s]*000$ if $x_k=1$, otherwise let $\sigma_k[s+1]=\sigma_k[s]*110$.
\item[(iv)] If $k=i_u$ for some $u\leq u_0$ then let $\sigma_k[s+1]=\sigma_k[s]*001$ if $x_k=1$, otherwise let $\sigma_k[s+1]=\sigma_k[s]*100$.
\end{itemize}
This concludes the actions at stage $s+1$. If
\begin{align*}
	& \Psi(\sigma_1[s]*01^m,\sigma_2[s]*01^m,\dots,\sigma_n[s]*01^m,0)=1\\
	\text{and}\quad & \Psi(\sigma_1[s]*01^m,\sigma_2[s]*01^m,\dots,\sigma_n[s]*01^m,1)=0\phantom{\text{and}\quad}
\end{align*}
we work with a tuple $(x_1,\dots,x_n)$ such that $G(x_1,\dots,x_n)\neq 1-F^{\vec{b}}_{\vec{i}}(x_1,\dots,x_n)$ instead.

\medskip

Now we check that, for any quadruple $(G,F,\Phi,\Psi)$, the pair of Turing functionals $\Phi,\Psi$ does not witness $s_{\vec{\beta}\uh n}^G \leq_{W} s_{\vec{\beta}\uh n}^F$. Suppose that they do, and suppose that $(G,F,\Phi,\Psi)$ is handled by the construction at stage $s+1$. Let $\vec{\sigma}[s]=(\sigma_1[s]*01^\omega,\dots,\sigma_n[s]*01^\omega)$. Some $m$ satisfying (i) must be found as $\Phi\left(\vec{\sigma}[s]\right)$ must output a tuple of strings of infinite length. If no $m$ satisfying (ii) exists then by construction we would have chosen each $\beta_k$ so that $s_{\beta_k}(\sigma_k[s]*01^\omega)=x_k$. Thus, 
\begin{align*}
	& s_{\vec{\beta}\uh n}^G\left(\vec{\sigma}[s]\right)=G(x_1,\dots,x_n)=b\neq \Psi\left(\vec{\sigma}[s],0\right)\\
	\text{and}\quad & s_{\vec{\beta}\uh n}^G\left(\vec{\sigma}[s]\right)\neq \Psi\left(\vec{\sigma}[s],1\right)\!. \phantom{\text{and}\quad}
\end{align*}
Therefore, in this case, the post-processing $\Psi$ does not work correctly on the input $\vec{\sigma}[s]$, which is a contradiction.

Thus we can assume that some least $m$ satisfying (i) and (ii) exists at stage~$s+1$. W.l.o.g.\ we can also assume $\Psi\left(\vec{\sigma}[s],0\right)=0$ and ${\Psi\left(\vec{\sigma}[s],1\right)=1}$; in the other case, where the post-processing flips the bit, we can argue the same way as below except with $1-F^{\vec{b}}_{\vec{i}}(x_1,\dots,x_n)$ in place of $F^{\vec{b}}_{\vec{i}}(x_1,\dots,x_n)$. 

By (iii)~and~(iv), for any $k\leq n$ we have $s_{\beta_k}\left(\sigma_k[s]*01^\omega\right)=x_k$. However, in the case where $k= i_u$ for some $u\leq u_0$, $b_u=0$ implies that $\tau_{k}$ is to the left of $\sigma_{k}[s]*001$ which implies that $\tau_k$ is left of $\beta_k$. Similarly if $b_u=1$ then $\tau_k$ is to the right of $\beta_k$, which means that $s_{\beta_{i_u}}(\tau_{i_u})=b_u$. On the other hand, if $k\neq i_u$ for any $u$, then, as $\tau_k$ is sufficiently long, it must be to the right of $\sigma_k[s]*000$ and also to the left of $\sigma_k[s]*110$. This means that $x_k=1$ implies that $\beta_k$ is left of $\tau_k$ and if $x_k=0$ then $\beta_k$ is right of $\tau_k$ and therefore for such a $k$ we have $s_{\beta_{k}}(\tau_{k})=x_k$.

We have $s_{\vec{\beta}\uh n}^F(\tau_1,\dots,\tau_n)=F(s_{\beta_1}(\tau_1),\dots,s_{\beta_n}(\tau_n))$, and since $k=i_u$ implies that ${s_{\beta_{k}}(\tau_{k})=b_u}$ while otherwise $s_{\beta_{k}}(\tau_{k})=x_k$, we conclude that 
\[F(s_{\beta_1}(\tau_1),\dots,s_{\beta_n}(\tau_n))=F^{\vec{b}}_{\vec{i}}(x_1,\dots,x_n)\neq G(x_1,\dots,x_n)=s_{\vec{\beta}\uh n}^G(\vec{\sigma}[s]),\]
which means that the post-processing does not act correctly on the input $\vec{\sigma}[s]$ which is a contradiction.

Now to observe that the inputs can be made $1$-equivalent, note that every step of the construction decides the next three bits of each $\beta_k$. Therefore we can use every fourth bit of $\beta_k$ to encode one bit of some other $\beta_j$. More precisely, define $\beta_k(4\langle j,m\rangle)=\beta_j(m)$ for every $k,j,m$; this is well-defined as $m<4\langle j,m\rangle$. Then $\beta_j\equiv_1 \beta_k$ for every $j,k$.
\end{proof}

\section{Rich truth-tables vs.\ multiple single-dimensional instances}\label{jqewhfjksadjkvf}

Let $\oplus_n$ denote addition modulo $2$ of $n$ bits. It seems intuitively clear that from a structural perspective $\oplus_{n}$ is strictly less ``rich'' than  $\oplus_{n+1}$, and that that should be true independently of the threshold values used. This intuition is confirmed by the following statement.
\begin{Theorem}\label{fjhdshjfhjasjh}
	For each $n>0$ and all proper $\alpha_1,\dots,\alpha_{n+1} \in 2^\omega$ we have 	
	\[s^{\oplus_{n}}_{\alpha_1,\dots,\alpha_n} \srwn s^{\oplus_{n+1}}_{\alpha_1,\dots,\alpha_{n+1}},\]
	where we write $A \srwn\! B$ to mean that both $A \leq_{sW}\! B$ and $B \nleq_{W}\! A$ hold.
\end{Theorem}
While the reducibility part of the statement is trivially true, the strictness is a corollary of the the next theorem combined with the obvious fact that 
\[s^{\oplus_{n}}_{\alpha_1,\dots,\alpha_{n}}\leq_{sW} s_{\alpha_1}\times\dots\times s_{\alpha_n}\]
for all $\alpha_1,\dots,\alpha_{n+1} \in 2^\omega$. It is natural to wonder how the resulting non-collapsing hierarchy relates to the hierarchy 
\[s_\alpha \srwn s_\alpha \times s_\alpha \srwn (s_\alpha)^3 \srwn \dots.\]
The answer is provided by the following results.
\begin{Theorem}\label{dsfugewahdfawe}
For each $n>0$ and all proper $\alpha_1,\dots,\alpha_{n+1},\alpha'_1,\dots, \alpha'_n\in 2^\omega$ we have
	\[s^{\oplus_{n+1}}_{\alpha_1,\dots,\alpha_{n+1}}\nleq_W s_{\alpha'_1}\times\dots\times s_{\alpha'_n}.\]
\end{Theorem}
\begin{proof}
Fix $n>0$ and proper sequences ${\alpha_1,\dots,\alpha_{n+1},\alpha'_1,\dots \alpha'_n \in 2^\omega}$. Suppose that $ s^{\oplus_{n+1}}_{\alpha_1,\dots,\alpha_{n+1}}\leq_W s_{\alpha'_1}\times\cdots\times s_{\alpha'_n}$ via the pre-processing $\Phi$ and post-processing $\Psi$. For any sequence of inputs $(\delta_1,\dots,\delta_n)$, we let $t\left(\delta_1,\dots,\delta_n\right)$ denote $\sum_{j=1}^n s_{\alpha'_j}(\delta_j)$.

Since $t\left(\Phi(\alpha_1,\dots,\alpha_{n+1})\right)$ is at most $n$ and there are $n+1$ many coordinates in the input to $\Phi$, there must be a $1 \leq i\leq n+1$ such that there is a sequence $\beta_{i+1},\beta_{i+2},\dots,\beta_{n+1}$ with the property that for every choice of~$\beta_i\in 2^\omega$,
\begin{align*}
&t\left(\Phi(\alpha_1,\dots,\alpha_{i-1},\beta_i,\beta_{i+1},\dots,\beta_{n+1})\right)\\
\geq\;&t\left(\Phi(\alpha_1,\dots,\alpha_{i-1},\alpha_i,\beta_{i+1},\dots,\beta_{n+1})\right).
\end{align*}
Fix the largest such $i$ and a corresponding sequence $\beta_{i+1},\dots,\beta_{n+1}$. Write 
\[(\delta_1,\dots,\delta_n)=\Phi(\alpha_1,\dots,\alpha_{i-1},\alpha_i,\beta_{i+1},\dots,\beta_{n+1})\] 
and fix a length $u$ long enough such that $\Phi(\alpha_1\uh u,\dots,\alpha_i\uh u,\dots,\beta_{n+1}\uh u)$ produces an output sequence of initial segments of $\delta_1,\dots,\delta_n$ such that for each~$j\leq  n$, if $\delta_j\neq \alpha'_j$, then the initial segment of $\delta_j$ produced is long enough to witness that inequality. 

We furthermore assume that $u$ is long enough so that 
	\begin{align*}
		&\Psi\left(  \alpha_1\uh u,\dots,\alpha_i\uh u,\beta_{i+1}\uh u,\dots,\beta_{n+1}\uh u, (s_{\alpha'_1}\times\cdots\times s_{\alpha'_n})(\delta_1,\dots,\delta_n)\right){\downarrow}\\
		=\;&	s^{\oplus_{n+1}}_{\alpha_1,\dots,\alpha_{n+1}}(\alpha_1,\dots,\alpha_{i-1},\alpha_i,\beta_{i+1},\dots,\beta_{n+1}).
	\end{align*}
As $\alpha_i$ is proper we can pick some $\beta_i$ with $\beta_i<_{lex}\alpha_i$ and ${\beta_i\uh u=\alpha_i\uh u}$. Then 
\begin{align*}
&s^{\oplus_{n+1}}_{\alpha_1,\dots,\alpha_{n+1}}(\alpha_1,\dots,\alpha_{i-1},\beta_i,\beta_{i+1},\dots,\beta_{n+1})\\
\neq\;& s^{\oplus_{n+1}}_{\alpha_1,\dots,\alpha_{n+1}}(\alpha_1,\dots,\alpha_{i-1},\alpha_i,\beta_{i+1},\dots,\beta_{n+1}).
\end{align*}
Now consider the pre- and post-processing on the input sequence
\[(\alpha_1,\dots,\alpha_{i-1},\beta_i,\beta_{i+1},\dots,\beta_{n+1}).\] 
By the choice of $i$ and of $\beta_{i+1},\dots,\beta_{n+1}$,
\begin{align*}
	&t\left(\Phi(\alpha_1,\dots,\alpha_{i-1},\beta_i,\beta_{i+1},\dots,\beta_{n+1})\right)\\
	\geq\;& t\left(\Phi(\alpha_1,\dots,\alpha_{i-1},\alpha_i,\beta_{i+1},\dots,\beta_{n+1})\right). 
\end{align*}
Write $(\delta_1',\dots,\delta_n')=\Phi(\alpha_1,\dots,\alpha_{i-1},\beta_i,\beta_{i+1},\dots,\beta_{n+1})$. By the choice of $u$, and the fact that $\beta_i\sqsupset\alpha_i\uh u$, we see that $s_{\alpha_j'}(\delta_j')\leq s_{\alpha_j'}(\delta_j)$ 
for each $j\leq n$. Thus $s_{\alpha_j'}(\delta_j')= s_{\alpha_j'}(\delta_j)$ for every $j\leq n$, and we have 
\begin{align*}
	&(s_{\alpha'_1}\times\cdots\times s_{\alpha'_n})\left(\Phi(\alpha_1,\dots,\alpha_{i-1},\beta_i,\beta_{i+1},\dots,\beta_{n+1})\right)\\
	=\;&(s_{\alpha'_1}\times\cdots\times s_{\alpha'_n})\left(\delta_1,\dots,\delta_n\right).
\end{align*}
Again by the choice of $u$ and the equality in the preceding line, we have
\begin{align*}
&\Psi\Big(  \alpha_1,\dots,\alpha_{i-1},\beta_i,\beta_{i+1},\dots,\beta_{n+1},\\[-0.5em]
&\qquad\qquad\quad(s_{\alpha'_1}\times\cdots\times s_{\alpha'_n})\left(\Phi(\alpha_1,\dots,\alpha_{i-1},\beta_i,\beta_{i+1},\dots,\beta_{n+1})\right)\!\Big)\\
=\;&\Psi\left(  \alpha_1\uh u,\dots,\alpha_i\uh u,\beta_{i+1}\uh u,\dots,\beta_{n+1}\uh u , (s_{\alpha'_1}\times\cdots\times s_{\alpha'_n})(\delta_1,\dots,\delta_n)\right)\\
=\;&s^{\oplus_{n+1}}_{\alpha_1,\dots,\alpha_{n+1}}(\alpha_1,\dots,\alpha_{i-1},\alpha_i,\beta_{i+1},\dots,\beta_{n+1})\\
\neq\;& s^{\oplus_{n+1}}_{\alpha_1,\dots,\alpha_{n+1}}(\alpha_1,\dots,\alpha_{i-1},\beta_i,\beta_{i+1},\dots,\beta_{n+1}).
\end{align*}
This is a contradiction since then the post-processing does not act correctly on  input~$(\alpha_1,\dots,\alpha_{i-1},\beta_i,\beta_{i+1},\dots,\beta_{n+1})$.
\end{proof}
In the other direction we even obtain a significantly stronger separation.
\begin{Theorem}
For each $n>0$ and all proper $\alpha_1,\dots,\alpha_{n},\beta,\gamma\in 2^\omega$ we have 
\[s_\beta \times s_\gamma \nleq_W s^{\oplus_{n}}_{\alpha_1,\dots,\alpha_n}.\]
\end{Theorem}
\begin{proof}
Assume otherwise. Fix $n>0$ and proper ${\alpha_1,\dots,\alpha_{n},\beta,\gamma\in 2^\omega}$ such that $s_\beta \times s_\gamma \leq_W s^{\oplus_{n}}_{\alpha_1,\dots,\alpha_n}$ via $(\Phi,\Psi)$. Write~$b$ for $s^{\oplus_{n}}_{\alpha_1,\dots,\alpha_n}(\Phi(\beta,\gamma))$ and fix~$n$ large enough so that $\Psi(\beta\uh n, \gamma \uh n,b)$ is defined; such an~$n$ must exist by assumption and we must have $\Psi(\beta\uh n, \gamma \uh n,b)=(1,1)$.  Now extend $\beta \uh n$ to some ${\beta' < \beta}$; this is possible because $\beta$ is proper. Fix $n'$ large enough so that ${\Psi(\beta'\uh n', \gamma \uh n',1-b)}$ is defined; again, such an~$n'$ exists by assumption and satisfies
${\Psi(\beta'\uh n', \gamma \uh n',1-b)}=(0,1)$. Finally extend $\gamma \uh n'$ to some $\gamma' < \gamma$.
By construction,
$\Psi(\beta',\gamma',b) =\Psi(\beta\uh n, \gamma \uh n,b)=(1,1)$;
this is an incorrect output since $s_\beta \times s_\gamma (\beta',\gamma')=(0,0)$.
\end{proof}
\goodbreak
We summarize our results in the following corollary and in Figure~\ref{fdsjkhasdjhkfdgjnasdhfsd}.
\begin{Corollary}
	The following results hold for all $n\geq 2$:
	\begin{itemize}
		\item[(i)] $s^{\oplus_n}_\alpha \nleq_W (s_\alpha)^{n-1}$.
		\item[(ii)] $ s^{\oplus_{n}}_\alpha \srwn (s_\alpha)^n$; in fact, even $s_\alpha \times s_\alpha \nleq_W s^{\oplus_{n}}_\alpha$.
	\end{itemize}
\end{Corollary}

\begin{figure}[tb]
	\begin{center}
		\begin{tikzpicture}[scale=.35,auto=left,every node/.style={fill=black!15},y=-2cm]
			\input{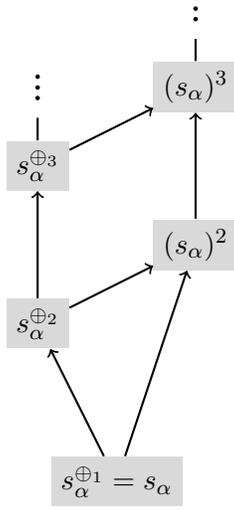}
		\end{tikzpicture}
		\caption{\label{fdsjkhasdjhkfdgjnasdhfsd}Products and parity truth tables; all reductions that are true in general are drawn.}
		\label{fig:diagramrreert}
	\end{center}
\end{figure}

The results above contrast with the following observation about the relation between the parity operation and single dimension step functions with multiple discontinuity points.
\begin{Proposition}
	For all $n\in \omega$ and $(\alpha_1,\dots,\alpha_n) \in (2^\omega)^n$ with 
	\[0 < \alpha_1 < \alpha_2 < \dots < \alpha_n < 1\]
	we have
	\[s_{(\alpha_1,\dots,\alpha_n)} <_W s^{\oplus_n}_{\alpha_1,\dots,\alpha_n}.\]
\end{Proposition}
\begin{proof}
	To see the reducibility, choose finite binary sequences such that
		\[\alpha_1 < q^\ell_1 <q^u_1 < \alpha_2 < q^\ell_2 <q^u_2 <\alpha_3 < \dots < \alpha_{n-1} < q^\ell_{n-1} <q^u_{n-1} < \alpha_n.\]
	The pre-processing $\Phi$ inspects an input $x$ until it detects
	\[x < q_1^u \vee \exists i\colon q^\ell_i < x < q^u_{i+1} \vee q_{n-1}^\ell < x,\]
	and then
	\begin{itemize}
		\item 	outputs $(x,0^\omega,\dots,0^\omega)$, if $x < q_1^u$ holds;
		\item 	outputs $(0^\omega,\dots,0^\omega,x)$, if $x > q_{n-1}^\ell$ holds;
		\item 	otherwise outputs $(0^\omega,\dots,0^\omega,\underset{\substack{\uparrow\\\makebox[0pt][c]{\scriptsize $(i+1)$-th component}}}{x},0^\omega,\dots,0^\omega)$;
	\end{itemize}
	in case that more than one case applies, $\Phi$ arbitrarily uses whichever it detects first.
		
	On input $(x,b)$, the post-processing $\Psi$ initially inspects the original input~$x$ to detect which case the pre-processing has picked; let the $k$-th component be the one into which $\Phi$ has chosen to output $x$. Then $\Psi$ outputs $k-1+b$.
	
	\medskip
	
Now we prove the strictness. We shall in fact prove that $s_{\beta,\gamma}^{\oplus_2}\nleq_W s_{(\alpha_1,\dots,\alpha_n)}$ for any proper $\beta,\gamma\in 2^\omega$. Assume towards a contradiction that there are proper strings $\beta$ and $\gamma$ such that $s_{\beta,\gamma}^{\oplus_2}\leq_W s_{(\alpha_1,\dots,\alpha_n)}$ is witnessed by the pre-processing~$\Phi$ and post-processing $\Psi$.

\smallskip

We first assume that $\Phi(\beta,\gamma)=\alpha_i$ for some $i\leq n$. Then there is some large enough length $u$ so that the output of $\Phi(\beta\uh u,\gamma\uh u)$ is extended by only $\alpha_i$ and not by any other $\alpha_j$. We also assume that $u$ is long enough so that $\Psi(\beta\uh u,\gamma\uh u, i)=0$, where we recall that $i=s_{(\alpha_1,\dots,\alpha_n)}(\Phi(\beta,\gamma))$. Now extend $\beta\uh u$ to some $\beta'<_{lex}\beta$, which is possible since $\beta$ is proper. Since $\Phi(\beta',\gamma)$ extends $\Phi(\beta\uh u,\gamma\uh u)$, we must have 
\[\alpha_{i-1}<_{lex}\Phi(\beta',\gamma)<_{lex}\alpha_{i+1},\]
where by $\alpha_0$ we mean $0$ and by $\alpha_{n+1}$ we mean $1$. Clearly we cannot have $\Phi(\beta',\gamma)=\alpha_i$ or $\Phi(\beta',\gamma)>_{lex}\alpha_i$, since otherwise $s_{(\alpha_1,\dots,\alpha_n)}(\Phi(\beta',\gamma))=i$ and the post-processing would not act correctly on input~$(\beta',\gamma)$. Thus we must have $\Phi(\beta',\gamma)<_{lex} \alpha_i$. Fix some length $u'$ long enough so that $\Phi(\beta'\uh u',\gamma\uh u')$ is extended by none of $\alpha_1,\dots,\alpha_n$ and such that $\Psi(\beta'\uh u',\gamma\uh u',i-1)=1$. Now finally extend $\gamma\uh u'$ to a string $\gamma'<_{lex}\gamma$; this is again possible since $\gamma$ is proper. Observe that on the one hand $s_{\beta,\gamma}(\beta',\gamma')=0$, but on the other hand $\Phi(\beta',\gamma')$ extends $\Phi(\beta'\uh u',\gamma\uh u')$ and therefore $\alpha_{i-1}<_{lex}\Phi(\beta',\gamma')<_{lex}\alpha_i$. Then~${s_{(\alpha_1,\dots,\alpha_n)}(\Phi(\beta',\gamma'))=i-1}$ and thus 
\[\Psi(\beta',\gamma',s_{(\alpha_1,\dots,\alpha_n)}(\Phi(\beta',\gamma')))=\Psi(\beta'\uh u',\gamma\uh u',i-1)=1\neq s_{\beta,\gamma}(\beta',\gamma').\]

\smallskip

For the other case where ${\Phi(\beta,\gamma)\not\in\{\alpha_1,\dots,\alpha_n\}}$ we argue similarly.
\end{proof}

\section{Multidimensional step functions with identical thresholds}\label{hhdashdfhbjshjehjchjasd}

In this section, we will study the case where all thresholds are identical; as a result, we will see a structurally much more interesting picture emerge as a result than in Section~\ref{stepintwodim}.
The following two observations about the two-dimensional case give a first glimpse of why this change is important; we only sketch the proofs as the results are corollaries of a much more general statement formally proven below.
\begin{Proposition}
	For	$\alpha \in 2^\omega$ we have $s^{\vee}_{\alpha,\alpha}\equiv_{sW}s^{\wedge}_{\alpha,\alpha}\equiv_{sW} s_\alpha$.
\end{Proposition}
\begin{proof}[Proof sketch]
	It is easy to see that to show ${s^{\vee}_{\alpha,\alpha}\leq_{sW} s_\alpha}$ it is sufficient to use ${(x,y) \mapsto \max\{x,y\}}$ as the pre-processing function and the identity as the post-processing. Similarly, for ${s^{\wedge}_{\alpha,\alpha}\leq_{sW} s_\alpha}$, use $(x,y) \mapsto \min\{x,y\}$. To see $s_\alpha \leq_{sW} s^{\vee}_{\alpha,\alpha},\, s^{\wedge}_{\alpha,\alpha}$, simply use the pre-processing $x \mapsto (x,x)$.
\end{proof}
\begin{Proposition}
	For	$\alpha \in 2^\omega$ we have $s^{\rightarrow}_{\alpha,\alpha}\equiv_{sW}s^{\leftarrow}_{\alpha,\alpha}\equiv_{sW} s^{\leftrightarrow}_{\alpha,\alpha}$.
\end{Proposition}
\begin{proof}[Proof sketch]
	The post-processing for all reductions is again the identity function. Then to show  $s^{\leftrightarrow}_{\alpha,\alpha}\leq_{sW} s^{\rightarrow}_{\alpha,\alpha}$, we can use 
	\[(x,y) \mapsto (\max\{x,y\}, \min\{x,y\})\] as the pre-processing.
	
	For $s^{\rightarrow}_{\alpha,\alpha}\leq_{sW} s^{\leftrightarrow}_{\alpha,\alpha}$, assume we are given input $(x,y)$. Initially we output $(x,y)$ unmodified. If at some point it becomes clear from inspecting longer and longer initial segments of the inputs that $x\neq y$, then we need to react: If $x < y$, then continue the previously generated outputs in such a way that they become $(x,x)$ in the limit; otherwise, that is if we see $x > y$, then continue in such a way as to produce $(x,y)$. It is easy to check that this produces the required behaviour.
	
	For $s^{\leftarrow}_{\alpha,\alpha}\equiv_{sW} s^{\leftrightarrow}_{\alpha,\alpha}$, the argument is symmetric.
\end{proof}
Thus, for every $\alpha$,
\[s_\alpha\equiv_{sW} s^{\vee}_{\alpha,\alpha}\equiv_{sW} s^{\wedge}_{\alpha,\alpha} \srwn s^{\leftrightarrow}_{\alpha,\alpha}\equiv_{sW} s^{\rightarrow}_{\alpha,\alpha}\equiv_{sW} s^{\leftarrow}_{\alpha,\alpha},\tag{$\ast$}\label{sdfsdfsdfsdfsdfdsfsdfsdfdsfsdf}\]
where the strictness of the inequality will follow from the results below.
This seems to suggest that the ``richness level'' of the truth-table~$F$ influences the strength of the Weihrauch degree of $s_{\alpha,\alpha}^F$, and this is what we prove in a much more general form in the following. Thus, for the rest of the article we will study the case of equal thresholds while also generalising to higher dimensions.

Before we do so, however, we observe in the following lemma that all results that we will obtain below also hold for the slightly more general case where the thresholds $\alpha_1,\dots,\alpha_n$ are not actually required to be equal, but where they are only assumed to satisfy $s_{\alpha_i}\equiv_{sW} s_{\alpha_j}$ for all $i,j$. In the remainder of the article we choose to 
neglect this in order to simplify notations.
\begin{Lemma}
	Let $A$ be a subset of Cantor space such that for any $\alpha, \beta \in A$ we have $s_{\alpha}\equiv_{sW} s_{\beta}$, and let any $n$-dimensional truth table $F$ be given. Then for any
	$\alpha_1,\dots,\alpha_n\in A$ and any $\beta_1,\dots,\beta_n\in A$ we have
	$s_{\alpha_1, \dots, \alpha_n}^F\equiv_{sW}s_{\beta_1, \dots, \beta_n}^F$.
	
	In particular, 
we can write $s_\alpha^F$ instead of $s_{\alpha_1, \dots, \alpha_n}^F$, where $\alpha$ is chosen arbitrarily from $A$.
\end{Lemma}
\begin{proof}
	By assumption, for all $i$, we have  $s_{\alpha_i}\leq_{sW} s_{\beta_i}$ via some pre-processing function~$\Phi_i$ and by Fact~\ref{fact:boundary} all post-processing functions are the identity. Then we can use $\Phi_1 \times \dots \times \Phi_n$ as the pre-processing and the identity as the post-processing function to see $s_{\alpha_1, \dots, \alpha_n}^F\leq_{sW}s_{\beta_1, \dots, \beta_n}^F$. The other direction is symmetric.
\end{proof}
So fix any proper $\alpha \in 2^\omega$; we will now study $s_\alpha^F$ in full generality.

\begin{Definition}
For $v\in 2^n$ we write $t(v)$ for the number of $j$ such that $v(j)=1$.
	
For $v,v'\in 2^n$, we write $v\subseteq v'$ if for every $j<n$ we have that $v(j)=1$ implies $v'(j)=1$. Write $v\subset v'$ if $v\subseteq v'$ and $v\neq v'$.
We write $v\subset^1 v'$ if $v\subset v'$ and there is no $v''$ such that $v\subset v''\subset v'$.

Given a truth-table $F$, if $v_1,\dots,v_k \in 2^{n(F)}$ is the longest sequence  such that for every $i<k$, $v_i\subset v_{i+1}$ and $F(v_i)\neq F(v_{i+1})$, we write $l(F)$ for $k-1$.

Clearly, $l(F)\leq n(F)$. If $l(F)= n(F)$ we call $F$ {\em complete}.
\end{Definition}
Informally, $l(F)$ is the maximal possible number of bit flips occurring in the output of $F$ when we flip $F$'s input bits from $0$ to $1$ one by one. For a finite binary string $w$ with $|w|=n$ we will also write $l(w)$ for 
\[|\{1\leq i \leq n-1 \colon w(i) \neq w(i+1)\}|,\]
that is, for the number of changes between $0$ and $1$ that we see when we go through the string.

Recall that $\oplus_n$ is addition modulo $2$ of $n$ bits; it is easy to see that $\oplus_n$ is complete.

\begin{Definition}
If $n$ is clear from context, we call {\em level $i$} the set
\[L_i=\{v\in 2^n\colon t(v)=i\}.\]
For a truth-table~$F$ we say that its level $i \leq n(F)$ is {\em homogeneous} if $F$ is constant on $L_i$. If level~$i$ of $F$ is homogeneous for all $0 \leq i \leq n(F)$, then we call $F$ {\em homogeneous}.
\end{Definition}

\begin{figure}[htb]
	\begin{center}
		\begin{tikzpicture}[scale=.5,auto=left,every node/.style={fill=black!15},y=-2cm]
			\input{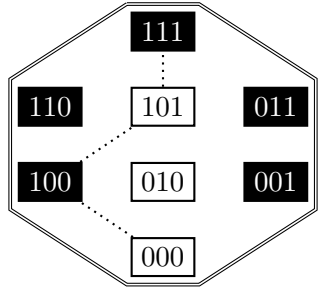}
		\end{tikzpicture}
		\caption{A truth table $F$ that is complete but inhomogeneous for $n=3$, with a path witnessing its completeness. A filled rectangle stands for combinations of input truth values that $F$ evaluates to true, an empty rectangle for those that are evaluated to false.}
		\label{fig:diagertertertram}
	\end{center}
\end{figure}
So $F$ being homogeneous means that $F$'s output depends solely on the number of $1$'s in its input, not on their positions. It is easy to see that $\oplus$ and $1-\oplus$ are the only truth-tables that are both homogeneous and complete.

The results in (\ref{sdfsdfsdfsdfsdfdsfsdfsdfdsfsdf}) can be phrased in terms of $l(F)$: Notice  ${l(\wedge)=l(\vee)=1}$ and $l(\rightarrow)=l(\leftarrow)=l(\leftrightarrow)=2$; thus, (\ref{sdfsdfsdfsdfsdfdsfsdfsdfdsfsdf}) seems to suggest that $l(F)=l(F')$ if and only if $s^F_\alpha\equiv_{sW} s^{F'}_\alpha$. We shall prove that this is indeed the case: the degree of $s^F_\alpha$ is completely determined by $l(F)$.
\begin{Theorem}\label{maintheorem}
Fix $n>1$ and an $n$-dimensional truth-table  $F$. Then \[s^F_\alpha\equiv_{sW} s^{\oplus_{l(F)}}_\alpha\!\!\!\!\!\!\!\!.\]
\end{Theorem}
Note that $\oplus_{l(F)}$ is an $l(F)$-dimensional truth-table, that is, a truth table possibly taking fewer bits as inputs than~$F$. Therefore the theorem immediately leads to the following corollary.
\begin{Corollary}\label{cor:equiv}
	Let $F$ be any Boolean function with $l(F)=1$ and fix any non-computable $\alpha$ and $\beta$. Then $s_\alpha\equiv_{sW} s_\beta$ if and only if $s_{\alpha,\alpha}^F\equiv_{sW} s_{\beta,\beta}^F$.
\end{Corollary}
Note further that the theorem implies that whether $s^F_\alpha\equiv_{sW} s^{F'}_\alpha$ depends solely on $F$ and~$F'$. This allows us to introduce the shorthand notations $F \leq_{sB} F'$ for $s^F_\alpha\leq_{sW} s^{F'}_\alpha$ and $F \leq_{B} F'$ for $s^F_\alpha\leq_{W} s^{F'}_\alpha$, as well as $F \srbn F'$ for $F \leq_{sB} F'$ and $F' \nleq_{B} F$. Then, when restricting to the case of identical thresholds,  Theorem~\ref{fjhdshjfhjasjh} becomes the following corollary; note that this implies that the hierarchy does not collapse.
\begin{Corollary}
	For each $n>1$, $\oplus_{n} \srbn \oplus_{n+1}$.
\end{Corollary}

The remainder of this section will be devoted to the proof of Theorem~\ref{maintheorem}. First we show a key lemma.

\begin{Lemma}\label{keylemma}
Let $F$ be a truth-table and $v\in 2^{n(F)}$ such that $t(v)<n(F)$. Assume that,
\begin{itemize}
	\item for every $w,w'\in 2^{n(F)}$ with $ w,w'\supseteq v$ and $t(w)=t(w')>t(v)$, we have $F(w)=F(w')$,
	\item and for every $w \in 2^{n(F)}$ with $ w\supseteq v$, if $t(w)=t(v)+1$ then ${F(w)=F(v)}$.
\end{itemize}
Then $F\leq_{sB} F'$, where $F'$ is defined for every $u\in 2^{n(F)}$ via	
\[F'(u)=\begin{cases}
1-F(u) & \text{if } u=v,\\
F(u) & \text{otherwise.} \\
\end{cases}\]
\end{Lemma}

Before we give the full proof, we give an informal explanation of the ideas involved for the case of a $4$-dimensional truth-table $F$ and $v=1100$.
In~ $F'$, the output of the truth-table at~$v$ will be flipped; however we want to still have $F\leq_{sB} F'$. We achieve this by reducing the case~$v$ to another case $v'$ which in $F'$ still exhibits the same behaviour as $v$ did in $F$.
All other cases should be reduced to themselves.
Thus, we need to describe a pre-processing that operates on input quadruples $(x_1,x_2,x_3,x_4)$ in such way that while it may change the concrete values of the $x_i$'s and in some cases even their position relative to $\alpha$, the final output after applying the truth-table remains the same.
Only in the case that $x_1$ and $x_2$ are at least $\alpha$ while $x_3$ and $x_4$ are strictly to the left of it, do we actually want to change the $x_i$'s in such a way that their position relative to $\alpha$ changes so that we reduce to the different case $v'$.

One criterion that we apply will be to check whether \[\max\{x_3,x_4\}<\min\{x_1,x_2\}.\tag{$\dagger$}\]
Clearly,  $(\dagger)$ is equivalent to being in a situation that is comparable with $v$ by relation~$\subseteq$ defined above; thus in particular it is a necessary but insufficient condition for being in case $v$.

Given some input $(x_1,x_2,x_3,x_4)$, we keep inspecting longer and longer initial segments of all input components. Initially they might all look the same. But note that, if an input separates from the others by exhibiting an initial segment that shows it to be smaller than the others, then the first input that does this must indeed be the smallest of the four inputs. Similarly, if one of the inputs veers upwards first, then it must be the largest of the four input components. Of course it is also possible that two inputs separate from the two others at the same time, looking the same for a while and potentially splitting again later.

Initially, we just output $\Phi(x_1,x_2,x_3,x_4)=(x_1,x_2,x_3,x_4)$ with all output component initial segments looking identical. After a while, the inputs might start separating from each other, however as long as $(\dagger)$ does not hold, we continue outputting the inputs unmodified. For $x_1$ and $x_2$ we will continue like this forever; that is, $\Phi$ will unconditionally leave its first two input components unmodified. As far as $x_3$ and $x_4$ are concerned, if at some point $(\dagger)$ turns true, then there are two cases:

\begin{itemize}
	\item[(i)] Either, $\min\{x_3,x_4\}$ looks strictly smaller than $\max\{x_3,x_4\}$ which in turn looks strictly smaller than $\min\{x_1,x_2\}$. Then the component that looks like $\min\{x_3,x_4\}$ will continue outputting $\min\{x_3,x_4\}$ forever (which means this input component remains unmodified by $\Phi$) and the other component starts outputting $\min\{x_1,x_2\}$; notice that this is possible because this latter component can only now begin to separate from $\min\{x_1,x_2\}$; otherwise $(\dagger)$ would have become true earlier.
	\item[(ii)] Or, $x_3$ and $x_4$ keep looking the same but now start looking strictly smaller than $\min\{x_1,x_2\}$. In this case, we arbitrarily let $x_3$ copy $\min\{x_1,x_2\}$ and leave $x_4$ unmodified. Again note that this is possible, because in this case both $x_3$ and $x_4$ have only now separated from $\min\{x_1,x_2\}$.
\end{itemize}
This completes the description of the pre-processing; let us consider the possible outcomes:
\begin{itemize}
	\item Should we be in the case $(s_\alpha(x_1),s_\alpha(x_2),s_\alpha(x_3),s_\alpha(x_4))=v$, then $x_1$ and $x_2$ were to the right of $\alpha$ and they remain so. One of $x_3$ and $x_4$, who both were to the left of $\alpha$, is replaced by $\min\{x_1,x_2\}$, which lies to the right of $\alpha$; the other is unchanged and stays to the left of $\alpha$. If we were in Case~(ii) above, the positions of the third and fourth components in the output quadruple might also be swapped. But in any case, from the truth values $v=1100$, we transition to a $v'\in \{1110, 1101\}$. Then, by the second assumption in Lemma~\ref{keylemma}, we have $F'(v')=F(v')=F(v)$, as required.
	
	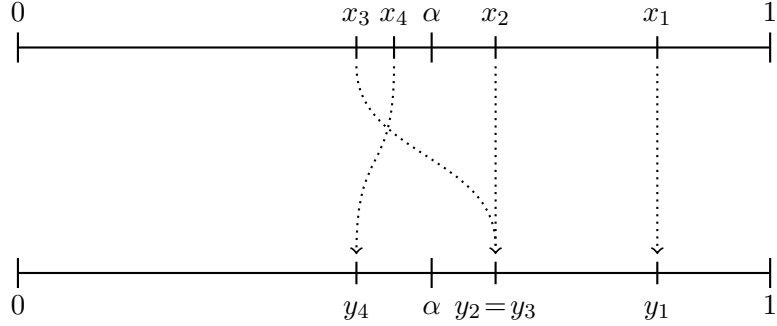
\begin{figure}[htb]
		\begin{center}
			\begin{tikzpicture}[scale=.5,auto=left,every node/.style={fill=black!15},y=-2cm]
				\draw[thick] (0,0) -- (20,0);
				
				\draw[thick] (0,-0.2) -- (0,0.2);
				\node[fill=none] at (0,-0.45) {\strut{$0$}};
				
				\draw[thick] (20,-0.2) -- (20,0.2);
				\node[fill=none] at (20,-0.45) {\strut{$1$}};
				
				\draw[thick] (11,-0.2) -- (11,0.2);
				\node[fill=none] at (11,-0.45) {\strut{$\alpha$}};
				
				\draw[thick] (17,-0.15) -- (17,0.15);
				\node[fill=none] at (17,-0.45) {\strut{$x_1$}};
				
				\draw[thick] (12.7,-0.15) -- (12.7,0.15);
				\node[fill=none] at (12.7,-0.45) {\strut{$x_2$}};
				
				\draw[thick] (9,-0.15) -- (9,0.15);
				\node[fill=none] at (9,-0.45) {\strut{$x_3$}};
				
				\draw[thick] (10,-0.15) -- (10,0.15);
				\node[fill=none] at (10,-0.45) {\strut{$x_4$}};

				\draw[thick] (0,3) -- (20,3);
				
				\draw[thick] (0,3.2) -- (0,2.8);
				\node[fill=none] at (0,3.45) {\strut{$0$}};
				
				\draw[thick] (20,3.2) -- (20,2.8);
				\node[fill=none] at (20,3.45) {\strut{$1$}};
				
				\draw[thick] (11,3.2) -- (11,2.8);
				\node[fill=none] at (11,3.45) {\strut{$\alpha$}};
				
				\draw[thick] (17,3.15) -- (17,2.85);
				\node[fill=none] at (17,3.45) {\strut{$y_1$}};
				
				\draw[thick] (12.7,3.15) -- (12.7,2.85);
				\node[fill=none] at (12.7,3.45) {\strut{$y_2\!=\!y_3$}};
				
				\draw[thick] (9,3.15) -- (9,2.85);
				\node[fill=none] at (9,3.45) {\strut{$y_4$}};

				\draw[thick,dotted,->] (17,0.25) -- (17,2.75);
				\draw[thick,dotted,->] (12.7,0.25) -- (12.7,2.75);
				\draw[thick,dotted] (9,0.25) to[in=90,out=-90,looseness=1] (12.7,2.63);
				\draw[thick,dotted,->] (10,0.25) to[in=90,out=-90,looseness=1.5] (9,2.75);
				
			\end{tikzpicture}
			\caption{The pre-processing   operating on an input $(x_1,x_2,x_3,x_4)$ with $(s_\alpha(x_1),s_\alpha(x_2),s_\alpha(x_3),s_\alpha(x_4))=v$ to produce output $(y_1,y_2,y_3,y_4)$ while Case~(ii) applies. After the pre-processing one more input component lies to the right of $\alpha$ than did before.}
			\label{fig:diagramdfgdfg}
		\end{center}
	\end{figure}
	
	\item If 
	\[t(s_\alpha(x_1),s_\alpha(x_2),s_\alpha(x_3),s_\alpha(x_4))=t(v)\]
	but 
	\[(s_\alpha(x_1),s_\alpha(x_2),s_\alpha(x_3),s_\alpha(x_4))\neq v\]
	then clearly condition $(\dagger)$ can never be triggered, and thus	
	\[\Phi(x_1,x_2,x_3,x_4)=(x_1,x_2,x_3,x_4).\]
	\item If $t(s_\alpha(x_1),s_\alpha(x_2),s_\alpha(x_3),s_\alpha(x_4))<t(v)$, then even if $(\dagger)$ is triggered, replacing one of~$x_3$ and $x_4$, who both were to the left of $\alpha$, by $\min\{x_1,x_2\}$, which in this case also lies to the left of $\alpha$, does not change the inputs to the truth-table. Again, if we were in Case~(ii), some of the output components might have switched places; but, as this can only occur between components to the left of $\alpha$, it has no impact.
	\item If $t(s_\alpha(x_1),s_\alpha(x_2),s_\alpha(x_3),s_\alpha(x_4))>t(v)$
	there are two subcases:
	\begin{itemize}
		\item If $v \not\subset (s_\alpha(x_1),s_\alpha(x_2),s_\alpha(x_3),s_\alpha(x_4))$ then condition $(\dagger)$ cannot get triggered.
		\item Assume otherwise. Then $x_1$, $x_2$ and $\max\{x_3,x_4\}$ must all three be at least $\alpha$. Thus, replacing $\max\{x_3,x_4\}$ by $\min\{x_1,x_2\}$ does not change any truth values. Note, however, that if we were in Case~(ii) above, then we made an arbitrary choice which component copies $\min\{x_1,x_2\}$ and which one $\min\{x_3,x_4\}$; thus the order of the truth values of the third and fourth component might be interchanged. But by the first condition in the theorem that does not change the output of the truth table.
	\end{itemize}
\end{itemize}
This completes our example; we can now give the proof in full generality.
\begin{proof}
Write $T$ for $\{i\colon v(i)=1\}$ and $F$ for 	$\{i\colon v(i)=0\}$. 

On input~$(x_1,\dots,x_n)$, the pre-processing~$\Phi$ proceeds as follows to produce output~$(y_1,\dots,y_n)$: All $x_i$ for $i \in T$ are always copied into $y_i$ unmodified. For all other $x_i$ with $i \in F$, we initially also copy them unmodified into $y_i$, while waiting for condition
	\[\max\{x_i\colon i \in F\}<\min\{x_i\colon i \in T\}.\tag{$\ddagger$}\]
to start looking true. By this we mean that we inspect longer and longer initial segments of all the $x_i$ and check whether these initial segments already witness $(\ddagger)$.

As soon as  $(\ddagger)$ first looks true, we let $U$ be the set of $i \in F$ such that $x_i$ currently looks like $\max\{x_i\colon i\in F\}$. It is clear that $U$ is not empty. We arbitrarily pick some $i \in U$ and call it~$i^\ast$. Let $L$ denote $F \setminus U$.

Thus $x_{i^\ast}$ has just now started to look different from ${\min\{x_i\colon i \in T\}}$, so that we can let $y_{i^\ast}$ copy $\min\{x_i\colon i \in T\}$. If $\ell := |U \setminus \{i^\ast\}|$ is non-zero, then let all~$y_i$ for $i \in U \setminus \{i^\ast\}$ from now on copy the $\ell$ smallest numbers in $\{x_i\colon i \in U\}$ (a~set containing $\ell+1$ numbers). This is still possible, because up to this point all $x_i$ with $i \in U$ look the same.

For all $i \in L$, we continue to copy the input component $x_i$ into~$y_i$ unchanged. This completes the construction.

\medskip

We verify that this proves the theorem by distinguishing these cases:
\begin{itemize}
	\item In case $(s_\alpha(x_1),\dots,s_\alpha(x_n))=v$ it must be clearly be the case that $(\ddagger)$ will eventually look true. All components $x_i$ with $i \in T$ were at least $\alpha$ and since
	for these $i$ we have $x_i=y_i$ the same holds after applying $\Phi$. The component $x_{i^\ast}$ was to the left of $\alpha$ but is in the output replaced by $y_{i^\ast}=\min\{x_i\colon i \in T\}$, which is at least $\alpha$. All other $y_i$ are equal to the $|F|-1$ smallest numbers in $\{x_i\colon i \in F\}$, and are thus all to the left of $\alpha$ (here, some $i$-th components with $i \in U \setminus \{i^\ast\}$ may have been swapped with each other). 
	
	Thus,   $v':=(s_\alpha(y_1),\dots,s_\alpha(y_n))$ is the same as $v$, except for one additional $1$ in the $i^\ast$-th component. By the second condition in the theorem, $F'(v')=F(v')=F(v)$.
	
	\item If $t(s_\alpha(x_1),\dots,s_\alpha(x_n))=t(v)$ but $(s_\alpha(x_1),\dots,s_\alpha(x_n))\neq v$ then clearly condition $(\ddagger)$ can never start looking true, and thus	
	\[(y_1,\dots,y_n)=(x_1,\dots,x_n).\]
	
	\item In case that $t(s_\alpha(x_1),\dots,s_\alpha(x_n))<t(v)$, even if $(\ddagger)$ is triggered,
	then replacing $x_{i^\ast}$, which was to the left of $\alpha$, by $y_{i^\ast}=\min\{x_i\colon i \in T\}$, which in this case also lies to the left of $\alpha$, results in \[(s_\alpha(x_1),\dots,s_\alpha(x_n))=(s_\alpha(y_1),\dots,s_\alpha(y_n));\]
	here again, some $i$-th components with $i \in U \setminus \{i^\ast\}$ may have been swapped with each other.
		
	\item If $t(s_\alpha(x_1),\dots,s_\alpha(x_n))>t(v)$
	there are two subcases:
	\begin{itemize}
		\item If $v \not\subset (s_\alpha(x_1),\dots,s_\alpha(x_n))$ then condition $(\ddagger)$ can never start looking true, and thus	
		$(y_1,\dots,y_n)=(x_1,\dots,x_n)$.
		\item Assume otherwise, that is $v \subset (s_\alpha(x_1),\dots,s_\alpha(x_n))$.
		As before, some $i$-th components with $i \in U \setminus \{i^\ast\}$ may have been swapped with each other, and for those~$i$'s we may have ${s_\alpha(x_i)\neq s_\alpha(y_i)}$.
		But for all $i \in T$ we have $x_i=y_i$ and thus also 
		\[v \subset (s_\alpha(y_1),\dots,s_\alpha(y_n)).\]
		And since all numbers in $\{x_i\colon i \in T\} \cup \{\max\{x_i\colon i\in F\}\}$ must be at least $\alpha$, when during the construction $\max\{x_i\colon i\in F\}$ is replaced by ${y_{i^\ast}=\min\{x_i\colon i \in T\}}$ this ensures that
		\[t(s_\alpha(x_1),\dots,s_\alpha(x_n))=t(s_\alpha(y_1),\dots,s_\alpha(y_n)).\]
		By the first condition in the theorem, those two properties are enough to keep the output of the truth-table constant.\qedhere
	\end{itemize}
\end{itemize}
\end{proof}
Note that in the lemma above, we have $l(F)\leq l(F')$, as expected.

\begin{figure}[h]
	\begin{center}
		\begin{tikzpicture}[scale=.5,auto=left,every node/.style={fill=black!15},y=-2cm]
			\input{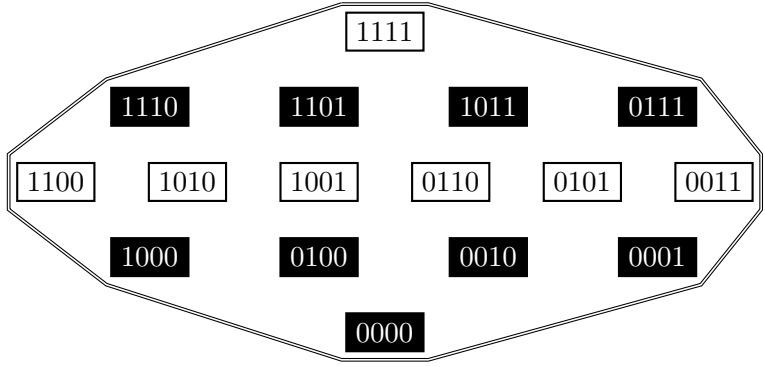}
		\end{tikzpicture}
		\caption{$H_{n,w}$ for $n=4$ and $w=11010$.}
		\label{fig:diagrertertertam}
	\end{center}
\end{figure}

Next, for each $n$ and each $l\leq n$ we identify several special Boolean functions~$F$ with $n(F)=n$ and $l(F)=l$.

Let $w\in 2^{n+1}$. We define $H_{n,w}$ to be the Boolean function with $n$ inputs and such that for every $v\in 2^n$, $H_{n,w}(v)=w(t(v))$. That is, whether $H_{n,w}$ outputs true only depends on the \textit{number} of true inputs; and which of those numbers lead to true or false is specified by the bits in $w$. It is immediate that for $w$ with $l(w)=l$ we have $l(H_{n,w})=l$.

Also define $K_{n,l}$
as $H_{n,w}$ for the word \[w= 0^{n-l} \; 0 \; 1 \; 0 \; 1 \;\dots \; (l-1 \,\mathrm{mod}\, 2) \; (l \,\mathrm{mod}\, 2).\]
That is, $K_{n,l}$ is homogeneous on every level, where on the $n-l$ lowest levels it always outputs false, while on the remaining  $l+1$ levels its output inverts every time we go up one level.

\begin{figure}[htb]
	\begin{center}
		\makebox[\textwidth][c]{
			\begin{tikzpicture}[scale=.5,auto=left,every node/.style={fill=black!15},y=-2cm]
				\input{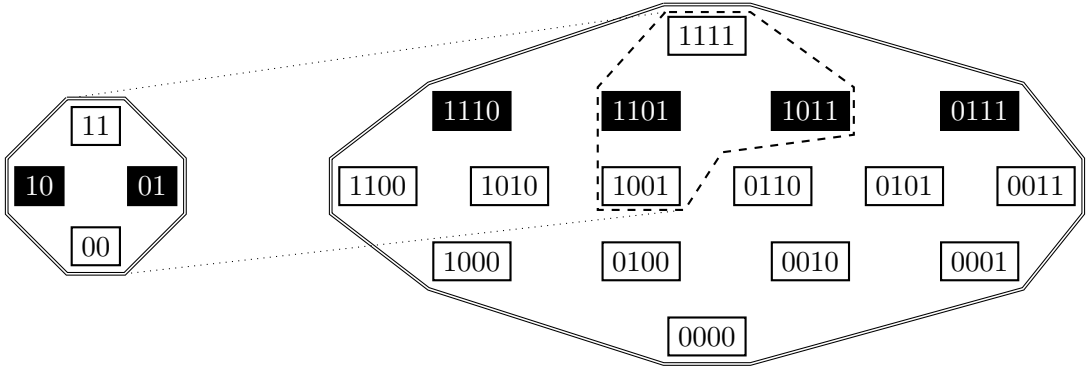}
			\end{tikzpicture}
		}
		\caption{$\oplus_l$ and $K_{n,l}$ for $l=2$ and $n=4$. One possible embedding of $\oplus_l$ into the highest three levels of $K_{n,l}$ is indicated. It is also easy to see that $l(\oplus_2)=l(K_{4,2})$.}
		\label{fig:diagrwefererteam}
	\end{center}
\end{figure}
\begin{Lemma}\label{lem:wtow}
Fix $k_0>0$ and let $w,w'\in 2^n$ and suppose that $w(k)=w'(k)$ for all $k> k_0$ and $w(k)=w'(k-1)$ for every $0<k\leq k_0$ and $w(0)=w'(0)$. Then 
\[s^{H_{n,w}}_\alpha\leq_{sW}  s^{H_{n,w'}}_\alpha\!\!\!\!\!\!\!\!\!.\]
\end{Lemma}
\begin{proof}
For an input $(x_1,\dots,x_n)$, write $\mu_1$ for the smallest number among the~$x_i$'s, $\mu_2$ for the second-smallest, etc. Then let the pre-processing output
\[\textstyle(y_1,\dots,y_n):=(\mu_1,\ldots,\mu_{n-k_0-1},\mu_{n-k_0},\mu_{n-k_0},\mu_{n-k_0+1},\ldots,\mu_{n-1}).\]
Write $k$ for $t(s_\alpha(x_1),\dots,s_\alpha(x_n))$ and $k'$ for $t(s_\alpha(y_1),\dots,s_\alpha(y_n))$. It is easy to check that if $k\geq k_0+1$ then $k'=k$; that if $0<k\leq k_0$ then $k'=k-1$; and finally, that if $k=0$ then~$k'=0$.
\end{proof}
Again, note that in the above, $l(H_{n,w})\leq l(H_{n,w'})$.

\begin{Lemma}\label{lem:wtow1}
Fix $k_0>0$ and let $w,w'\in 2^{n+1}$ and suppose that ${w(k)=w'(k)}$ for all~${k> k_0}$ and ${w(k)=w'(k+1)}$ for every~${k\leq k_0}$. Then 
\[s^{H_{n,w}}_\alpha\leq_W  s^{H_{n,w'}}_\alpha\!\!\!\!\!\!\!\!\!.\]
\end{Lemma}
\begin{proof}
For an input $(x_1,\dots,x_n)$ let the pre-processing output 
\[\textstyle(y_1,\dots,y_n):=(\mu_1,\ldots,\mu_{n-k_0-2},\mu_{n-k_0-1},\mu_{n-k_0+1},\mu_{n-k_0+2},\ldots,\mu_{n},1^\omega).\]
Again write $k$ for $t(s_\alpha(x_1),\dots,s_\alpha(x_n))$ and $k'$ for $t(s_\alpha(y_1),\dots,s_\alpha(y_n))$. It is easy to check that if $k\geq k_0+1$ then $k'=k$, and that if $0<k\leq k_0$ then $k'=k+1$.
\end{proof}

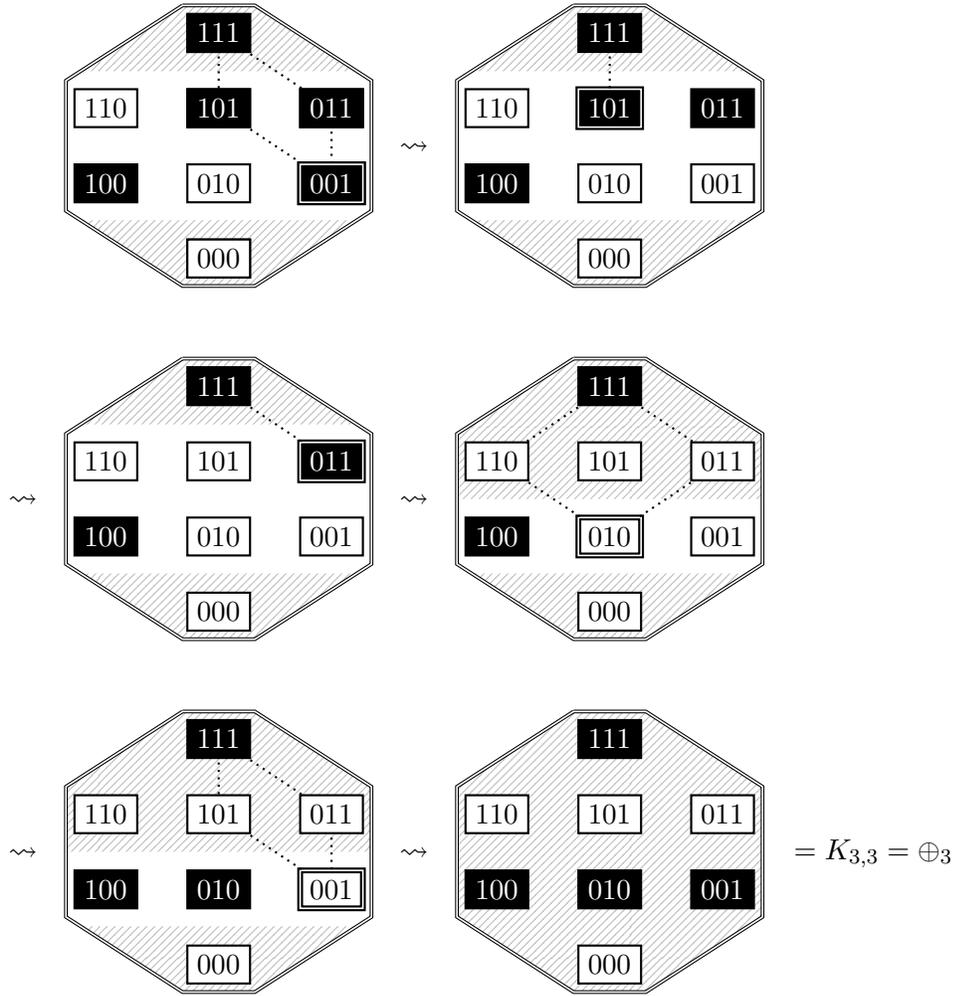
\begin{figure}[ht]
	\begin{center}
		\makebox[\textwidth][c]{

\begin{tabular}{ccccccc}

&$\vcenter{\hbox{\begin{tikzpicture}[scale=.5,auto=left,every node/.style={fill=black!15},y=-2cm]
	\tikzstyle{truth}=[color=white,fill=black,draw=black,thick]
	\tikzstyle{lie}=[color=black,fill=white,draw=black,thick]

\fill[pattern color=lightgray, pattern=north east lines]  ($(TFF.south west) !.125! (FFF.south west) + (-0.2,0.1)$) -- ($(FFF.south west)+ (-0.1,0.1)$) -- ($(FFF.south east)+ (0.1,0.1)$) -- ($(FFF.south east) !.875! (FFT.south east)+ (0.2,0.1)$);
	
		\fill[pattern color=lightgray, pattern=north east lines] ($(TTT.north west) + (-0.1,-0.1)$) -- ($(TTT.north west) !.875! (TTF.north west)  + (-0.15,-0.1)$)  -- ($(FTT.north east) !.125! (TTT.north east)  + (0.15,-0.1)$) -- ($(TTT.north east)+ (0.1,-0.1)$) --  ($(TTT.north west)+ (-0.1,-0.1)$);
	 
	\node[truth] (TTT) at (0,0) {$111$};
	
	\node[lie] (TTF) at (-3,1) {$110$};
	\node[truth] (TFT) at (0,1) {$101$};
	\node[truth] (FTT) at (3,1) {$011$};

	\node[truth] (TFF) at (-3,2) {$100$};
	\node[lie] (FTF) at (0,2) {$010$};
	\node[truth,double] (FFT) at (3,2) {$001$};
	
	\node[lie] (FFF) at (0,3) {$000$};
	
	\foreach \from/\to in {FFT/TFT,FFT/FTT,TFT/TTT,FTT/TTT}
	\draw [thick, dotted] (\from) -- (\to);  
	
	\draw[double] ($(TTT.north west) + (-0.1,-0.1)$) -- ($(TTF.north west)+ (-0.2,-0.1)$) -- ($(TFF.south west)+ (-0.2,0.1)$) -- ($(FFF.south west)+ (-0.1,0.1)$) -- ($(FFF.south east)+ (0.1,0.1)$) -- ($(FFT.south east)+ (0.2,0.1)$) -- ($(FTT.north east)+ (0.2,-0.1)$) -- ($(TTT.north east)+ (0.1,-0.1)$) --  ($(TTT.north west)+ (-0.1,-0.1)$);
	
	\node[truth] (TTT) at (0,0) {$111$};
	\node[truth] (TFT) at (0,1) {$101$};
	\node[truth] (FTT) at (3,1) {$011$};
	\node[truth,double] (FFT) at (3,2) {$001$};
	
	\node[lie] (FFF) at (0,3) {$000$};
\end{tikzpicture}}}$&

$\rightsquigarrow$

&$\vcenter{\hbox{\begin{tikzpicture}[scale=.5,auto=left,every node/.style={fill=black!15},y=-2cm]
			\tikzstyle{truth}=[color=white,fill=black,draw=black,thick]
			\tikzstyle{lie}=[color=black,fill=white,draw=black,thick] 

\fill[pattern color=lightgray, pattern=north east lines] ($(TTT.north west) + (-0.1,-0.1)$) -- ($(TTT.north west) !.875! (TTF.north west)  + (-0.15,-0.1)$)  -- ($(FTT.north east) !.125! (TTT.north east)  + (0.15,-0.1)$) -- ($(TTT.north east)+ (0.1,-0.1)$) --  ($(TTT.north west)+ (-0.1,-0.1)$);

\fill[pattern color=lightgray, pattern=north east lines]  ($(TFF.south west) !.125! (FFF.south west) + (-0.2,0.1)$) -- ($(FFF.south west)+ (-0.1,0.1)$) -- ($(FFF.south east)+ (0.1,0.1)$) -- ($(FFF.south east) !.875! (FFT.south east)+ (0.2,0.1)$);
			
\node[truth] (TTT) at (0,0) {$111$};

\node[lie] (TTF) at (-3,1) {$110$};
\node[truth,double] (TFT) at (0,1) {$101$};
\node[truth] (FTT) at (3,1) {$011$};

\node[truth] (TFF) at (-3,2) {$100$};
\node[lie] (FTF) at (0,2) {$010$};
\node[lie] (FFT) at (3,2) {$001$};

\node[lie] (FFF) at (0,3) {$000$};
			
			\foreach \from/\to in {TFT/TTT}
			\draw [thick, dotted] (\from) -- (\to);  
			
			\draw[double] ($(TTT.north west) + (-0.1,-0.1)$) -- ($(TTF.north west)+ (-0.2,-0.1)$) -- ($(TFF.south west)+ (-0.2,0.1)$) -- ($(FFF.south west)+ (-0.1,0.1)$) -- ($(FFF.south east)+ (0.1,0.1)$) -- ($(FFT.south east)+ (0.2,0.1)$) -- ($(FTT.north east)+ (0.2,-0.1)$) -- ($(TTT.north east)+ (0.1,-0.1)$) --  ($(TTT.north west)+ (-0.1,-0.1)$);
\end{tikzpicture}}}$&

\phantom{$\rightsquigarrow$}
\\\\\\

$\rightsquigarrow$

&$\vcenter{\hbox{\begin{tikzpicture}[scale=.5,auto=left,every node/.style={fill=black!15},y=-2cm]
			\tikzstyle{truth}=[color=white,fill=black,draw=black,thick]
			\tikzstyle{lie}=[color=black,fill=white,draw=black,thick] 
		
	\fill[pattern color=lightgray, pattern=north east lines]  ($(TFF.south west) !.125! (FFF.south west) + (-0.2,0.1)$) -- ($(FFF.south west)+ (-0.1,0.1)$) -- ($(FFF.south east)+ (0.1,0.1)$) -- ($(FFF.south east) !.875! (FFT.south east)+ (0.2,0.1)$);
			
\fill[pattern color=lightgray, pattern=north east lines] ($(TTT.north west) + (-0.1,-0.1)$) -- ($(TTT.north west) !.875! (TTF.north west)  + (-0.15,-0.1)$)  -- ($(FTT.north east) !.125! (TTT.north east)  + (0.15,-0.1)$) -- ($(TTT.north east)+ (0.1,-0.1)$) --  ($(TTT.north west)+ (-0.1,-0.1)$);
			
\node[truth] (TTT) at (0,0) {$111$};

\node[lie] (TTF) at (-3,1) {$110$};
\node[lie] (TFT) at (0,1) {$101$};
\node[truth,double] (FTT) at (3,1) {$011$};

\node[truth] (TFF) at (-3,2) {$100$};
\node[lie] (FTF) at (0,2) {$010$};
\node[lie] (FFT) at (3,2) {$001$};

\node[lie] (FFF) at (0,3) {$000$};
			
			\foreach \from/\to in {FTT/TTT}
			\draw [thick, dotted] (\from) -- (\to);  
			
			\draw[double] ($(TTT.north west) + (-0.1,-0.1)$) -- ($(TTF.north west)+ (-0.2,-0.1)$) -- ($(TFF.south west)+ (-0.2,0.1)$) -- ($(FFF.south west)+ (-0.1,0.1)$) -- ($(FFF.south east)+ (0.1,0.1)$) -- ($(FFT.south east)+ (0.2,0.1)$) -- ($(FTT.north east)+ (0.2,-0.1)$) -- ($(TTT.north east)+ (0.1,-0.1)$) --  ($(TTT.north west)+ (-0.1,-0.1)$);
\end{tikzpicture}}}$&

$\rightsquigarrow$

&$\vcenter{\hbox{\begin{tikzpicture}[scale=.5,auto=left,every node/.style={fill=black!15},y=-2cm]
			\tikzstyle{truth}=[color=white,fill=black,draw=black,thick]
			\tikzstyle{lie}=[color=black,fill=white,draw=black,thick]

\fill[pattern color=lightgray, pattern=north east lines] ($(TTT.north west) + (-0.1,-0.1)$) -- 
($(TTF.north west)+ (-0.2,-0.1)$) --
($(TTF.north west) !.5! (TFF.south west) + (-0.2,0)$)
-- ($(FFT.south east) !.5! (FTT.north east) + (0.2,0)$) 
-- ($(FTT.north east)+ (0.2,-0.1)$)
-- ($(TTT.north east)+ (0.1,-0.1)$) --  ($(TTT.north west)+ (-0.1,-0.1)$);
			
			\fill[pattern color=lightgray, pattern=north east lines]  ($(TFF.south west) !.125! (FFF.south west) + (-0.2,0.1)$) -- ($(FFF.south west)+ (-0.1,0.1)$) -- ($(FFF.south east)+ (0.1,0.1)$) -- ($(FFF.south east) !.875! (FFT.south east)+ (0.2,0.1)$);
\node[truth] (TTT) at (0,0) {$111$};

\node[lie] (TTF) at (-3,1) {$110$};
\node[lie] (TFT) at (0,1) {$101$};
\node[lie] (FTT) at (3,1) {$011$};

\node[truth] (TFF) at (-3,2) {$100$};
\node[lie,double] (FTF) at (0,2) {$010$};
\node[lie] (FFT) at (3,2) {$001$};

\node[lie] (FFF) at (0,3) {$000$};
			
			\foreach \from/\to in {FTF/TTF,FTF/FTT,TTF/TTT,FTT/TTT}
			\draw [thick, dotted] (\from) -- (\to);  
			
			\draw[double] ($(TTT.north west) + (-0.1,-0.1)$) -- ($(TTF.north west)+ (-0.2,-0.1)$) -- ($(TFF.south west)+ (-0.2,0.1)$) -- ($(FFF.south west)+ (-0.1,0.1)$) -- ($(FFF.south east)+ (0.1,0.1)$) -- ($(FFT.south east)+ (0.2,0.1)$) -- ($(FTT.north east)+ (0.2,-0.1)$) -- ($(TTT.north east)+ (0.1,-0.1)$) --  ($(TTT.north west)+ (-0.1,-0.1)$);
\end{tikzpicture}}}$

\\\\\\

$\rightsquigarrow$

&$\vcenter{\hbox{\begin{tikzpicture}[scale=.5,auto=left,every node/.style={fill=black!15},y=-2cm]
			\tikzstyle{truth}=[color=white,fill=black,draw=black,thick]
			\tikzstyle{lie}=[color=black,fill=white,draw=black,thick] 
			
	\fill[pattern color=lightgray, pattern=north east lines]  ($(TFF.south west) !.125! (FFF.south west) + (-0.2,0.1)$) -- ($(FFF.south west)+ (-0.1,0.1)$) -- ($(FFF.south east)+ (0.1,0.1)$) -- ($(FFF.south east) !.875! (FFT.south east)+ (0.2,0.1)$);
			
\fill[pattern color=lightgray, pattern=north east lines] ($(TTT.north west) + (-0.1,-0.1)$) -- 
($(TTF.north west)+ (-0.2,-0.1)$) --
($(TTF.north west) !.5! (TFF.south west) + (-0.2,0)$)
-- ($(FFT.south east) !.5! (FTT.north east) + (0.2,0)$) 
-- ($(FTT.north east)+ (0.2,-0.1)$)
-- ($(TTT.north east)+ (0.1,-0.1)$) --  ($(TTT.north west)+ (-0.1,-0.1)$);
			
			\node[truth] (TTT) at (0,0) {$111$};
			
			\node[lie] (TTF) at (-3,1) {$110$};
			\node[lie] (TFT) at (0,1) {$101$};
			\node[lie] (FTT) at (3,1) {$011$};

			\node[truth] (TFF) at (-3,2) {$100$};
			\node[truth] (FTF) at (0,2) {$010$};
			\node[lie,double] (FFT) at (3,2) {$001$};
			
			\node[lie] (FFF) at (0,3) {$000$};
			
			\foreach \from/\to in {FFT/TFT,FFT/FTT,TFT/TTT,FTT/TTT}
			\draw [thick, dotted] (\from) -- (\to);  
			
			\draw[double] ($(TTT.north west) + (-0.1,-0.1)$) -- ($(TTF.north west)+ (-0.2,-0.1)$) -- ($(TFF.south west)+ (-0.2,0.1)$) -- ($(FFF.south west)+ (-0.1,0.1)$) -- ($(FFF.south east)+ (0.1,0.1)$) -- ($(FFT.south east)+ (0.2,0.1)$) -- ($(FTT.north east)+ (0.2,-0.1)$) -- ($(TTT.north east)+ (0.1,-0.1)$) --  ($(TTT.north west)+ (-0.1,-0.1)$);

\end{tikzpicture}}}$&

$\rightsquigarrow$

&$\vcenter{\hbox{\begin{tikzpicture}[scale=.5,auto=left,every node/.style={fill=black!15},y=-2cm]
			\tikzstyle{truth}=[color=white,fill=black,draw=black,thick]
			\tikzstyle{lie}=[color=black,fill=white,draw=black,thick] 
	
	\fill[pattern color=lightgray, pattern=north east lines]  ($(TTT.north west) + (-0.1,-0.1)$) -- ($(TTF.north west)+ (-0.2,-0.1)$) -- ($(TFF.south west)+ (-0.2,0.1)$) -- ($(FFF.south west)+ (-0.1,0.1)$) -- ($(FFF.south east)+ (0.1,0.1)$) -- ($(FFT.south east)+ (0.2,0.1)$) -- ($(FTT.north east)+ (0.2,-0.1)$) -- ($(TTT.north east)+ (0.1,-0.1)$) --  ($(TTT.north west)+ (-0.1,-0.1)$);
			
			\node[truth] (TTT) at (0,0) {$111$};
			
			\node[lie] (TTF) at (-3,1) {$110$};
			\node[lie] (TFT) at (0,1) {$101$};
			\node[lie] (FTT) at (3,1) {$011$};

			\node[truth] (TFF) at (-3,2) {$100$};
			\node[truth] (FTF) at (0,2) {$010$};
			\node[truth] (FFT) at (3,2) {$001$};
			
			\node[lie] (FFF) at (0,3) {$000$};
			
			\foreach \from/\to in {}
			\draw [thick, dotted] (\from) -- (\to);  
			
			\draw[double] ($(TTT.north west) + (-0.1,-0.1)$) -- ($(TTF.north west)+ (-0.2,-0.1)$) -- ($(TFF.south west)+ (-0.2,0.1)$) -- ($(FFF.south west)+ (-0.1,0.1)$) -- ($(FFF.south east)+ (0.1,0.1)$) -- ($(FFT.south east)+ (0.2,0.1)$) -- ($(FTT.north east)+ (0.2,-0.1)$) -- ($(TTT.north east)+ (0.1,-0.1)$) --  ($(TTT.north west)+ (-0.1,-0.1)$);
\end{tikzpicture}}}$&

$=$\rlap{$\;K_{3,3}=\oplus_3$}

\end{tabular}
		}
		\caption{Three-dimensional example for the proof of Lemma~\ref{djhasdhsfdhjsdewuixvcnyd}. Each step corresponds to one application of Lemma~\ref{keylemma}, where the respective $v$ is marked by a double-lined border and its relevant extensions through dotted lines. The hatched areas mark the increasing restriction on possible $v$'s to select next. In the end, a truth-table is achieved that is homogeneous. Since we started with a complete truth-table, the resulting final truth-table must be complete as well.}
		\label{fig:diaertertergjghgram}
	\end{center}
\end{figure}
\begin{Lemma}For any $n\in\omega$, all $l\leq n$, every $w\in 2^{n+1}$ with $l(w)=l$, and any Boolean function $F$ such that  $n(F)=n$ and $l(F)=l$, we have \[\oplus_{l}\equiv_{sB} K_{n,l} \equiv_{sB} H_{n,w} \leq_{sB} F.\]
\end{Lemma}
\begin{proof}
To show $s^{\oplus_{l}}_\alpha\leq_W s^{K_{n,l}}_\alpha$ we just add some useless coordinates in the pre-processing that will always give true, that is, add $n-l$ entries $1$'s to the input tuple.

To show $s^{K_{n,l}}_\alpha\leq_W s^{\oplus_{l}}_\alpha$, just cut away some coordinates; that is, keep the $l$~many smallest values of the input tuple. 

To show $s^{K_{n,l}}_\alpha \leq_W s^{H_{n,w}}_\alpha$ we may have to ``spread out'' the output alternations that occur when changing levels; namely, the alternations in $s^{K_{n,l}}_\alpha$ are all concentrated at the top, while the alternations in $s^{H_{n,w}}_\alpha$ are determined by~$w$.
We apply Lemma \ref{lem:wtow} repeatedly to obtain $s^{K_{n,l}}_\alpha \leq_W s^{H_{n,w}}_\alpha$. To show $s^{H_{n,w}}_\alpha\leq_W s^{K_{n,l}}_\alpha$ we apply Lemma \ref{lem:wtow1} repeatedly.

Note that the results so far imply in particular that $H_{n,w} \equiv_B H_{n,w'}$ for all $w'$ with $l(w)=l(w')$. Thus for the last part of the proof we can w.l.o.g.\ replace $w$ by any such $w'$ that fits our needs.

So fix $F$. Since $l(F)=l$, fix a path $v_0\subset^1 v_1\subset^1 v_2\subset^1\cdots \subset^1 v_n$ with $v_i \in L_i$ for all~$i$ such that $w'= F(v_0)F(v_1)\cdots F(v_n)$ has $l(w')=l(F)=l$. Then $s^{H_{n,w'}}_\alpha\leq_W s_\alpha^F$ via a pre-processing that sends every $u$ to $v_{t(u)}$; such a pre-processing operates as follows: Follow the path from $v_0$ to $v_n$ and observe an additional~$1$ showing up with every step. The pre-processing will reorder its input tuple in such a way that the maximal input component will be put in the coordinate where the first $1$ shows up when going from $v_0$ to $v_1$, the second largest input component into the coordinate where the second $1$ shows up when going from $v_1$ to $v_2$, etc.
\end{proof}
Finally, to complete the proof of Theorem \ref{maintheorem}, we show the following.
\begin{Lemma}\label{djhasdhsfdhjsdewuixvcnyd}
	Fix $F$ with $l(F)=l$. Then $F\leq_{sB}  K_{n,l}$.
\end{Lemma}
\begin{proof}
Write $b$ for $F(1^n)$ and%
\def\bb{\overline{b}}
$\bb$ for $1-b$. We search for a $\subseteq$-minimal ${0^n\subset v\subset 1^n}$ such that for every $v'\supseteq v$ we have $F(v')=b$. If this $v$ exists, we apply Lemma \ref{keylemma} to obtain $s_\alpha^F\leq_W s_\alpha^{F'}$, where $F'(v)=\bb$ and $F'$~is the same as $F$ everywhere else. Note that $l(F)\leq l(F')$. However, by the minimality of $v$, and the fact that $v\neq 0^n$, we have $l(F)=l(F')$. Now repeat with $F'$ in place of $F$, each time flipping an output value from $b$ to $\bb$, until we are unable to find $v$. Note that $l(F)$ is preserved.

Now assume that we cannot find $v$ in $F$. Then any $v'$ with $t(v')=n-1$ must satisfy $F(v')=\bb$; this is because $F(1^n)=b$ by assumption and $1^n$ is the only extension of such a~$v'$; thus the existence of such a $v'$ with $F(v')=b$ would contradict the assumption that no $v$ as required exists.

Thus, assume that $F(v')=\bb$ for every $t(v')=n-1$; we proceed with the next layer and wish to make $F(v)=b$ for all $v$ with $t(v)=n-2$. We continue analogously, but this time at each sub-step we search for a $\subseteq$-minimal $v$ such that $0<t(v)\leq n-2$ and such that for every $v'\supseteq v$ with $t(v')<n$, we have $F(v')=\bb$. We can still apply Lemma \ref{keylemma} to $F$ and~$v$, and again by the minimality of $v$, and the fact that $v\neq 0^n$, we also have $l(F)=l(F')$. For the same reasons as before, once we cannot find $v$ anymore, we have achieved that $F(v)=b$ for every $t(v)=n-2$.

We continue this way, going down level by level, until we obtain the final form $K_{n,l}$.
\end{proof}
\goodbreak
\bibliographystyle{plain}
\bibliography{refs}

\begin{thebibliography}{1}

\bibitem{brattka_et_al:DagRep.8.9.1}
Vasco Brattka, Damir~D. Dzhafarov, Alberto Marcone, and Arno Pauly.
\newblock {Measuring the Complexity of Computational Content: From
  Combinatorial Problems to Analysis (Dagstuhl Seminar 18361)}.
\newblock {\em Dagstuhl Reports}, 8(9):1--28, 2019.

\bibitem{DBLP:journals/corr/BrattkaGP17}
Vasco Brattka, Guido Gherardi, and Arno Pauly.
\newblock Weihrauch complexity in computable analysis.
\newblock In Vasco Brattka and Peter Hertling, editors, {\em Handbook of
  Computability and Complexity in Analysis}, pages 367--417. Springer, 2021.

\bibitem{day2022three}
Adam Day, Rod Downey, and Linda Westrick.
\newblock Three topological reducibilities for discontinuous functions.
\newblock {\em Transactions of the American Mathematical Society, Series~B},
  9(28):859--895, 2022.

\bibitem{LinTakArn}
Takayuki Kihara, Arno Pauly, and Linda~Brown Westrick.
\newblock Articles in preparation.

\end{thebibliography}
\end{document}